\documentclass[a4study,twoside,11pt]{article}
\usepackage{amssymb}
\usepackage{amsmath}
\usepackage{graphicx}
\usepackage{amsthm}
\usepackage{cite}
\usepackage{array}

\usepackage{amsmath,latexsym,amssymb,amsfonts,amsbsy,amsthm}
\usepackage{graphicx}
\usepackage{subfigure}
\usepackage[a4paper]{geometry}
\usepackage{epstopdf}
\usepackage{diagbox}
\usepackage{enumerate}
\usepackage{color}

\usepackage{comment}
\usepackage{indentfirst}

\newfont{\bb}{msbm10}

\newtheorem{example}{Example}[section]

\newtheorem{theorem}{Theorem}[section]
\newtheorem{definition}{Definition}[section]
\newtheorem{lemma}{Lemma}[section]

\usepackage{verbatim}
\usepackage{algorithm}
\usepackage{algorithmicx}
\usepackage{algpseudocode}
\usepackage{amsmath}
\usepackage{graphics}
\usepackage{epsfig}
\usepackage{arydshln}

\usepackage{mathrsfs}
\usepackage{graphicx}
\usepackage{subfigure}
\usepackage{caption}
\usepackage{multirow}
\usepackage{threeparttable}
\numberwithin{equation}{section}

\usepackage{float}
\usepackage{footnote}
\makesavenoteenv{table}
\usepackage{booktabs}

\usepackage{authblk}

\usepackage{bm}
\usepackage{hyperref} 

\usepackage{newtxmath}    

\newenvironment{Proof}{{\noindent\it Proof.}}{\hfill $\square$ \par}

\title{A fast block nonlinear Bregman-Kaczmarz method with averaging for nonlinear sparse signal recovery}
 
\author{
	Aqin Xiao\\
	School of Mathematical Sciences, Tongji University, \\
	Shanghai, 200092, PR China. \\
	Email:xiaoaqin@tongji.edu.cn \\  
 Xiangyu Gao\\
 School of Mathematical Sciences, Tongji University, \\
  Shanghai, 200092, PR China. \\
Email:2211182@tongji.edu.cn \\
and\\
  Jun-Feng Yin\thanks{Corresponding author.}\\
	School of Mathematical Sciences, Tongji University,\\
 Key Laboratory of Intelligent Computing and Applications (Ministry of Education),\\
	Shanghai, 200092, PR China.\\
	Email:yinjf@tongji.edu.cn\\ 
}

\begin{document}
\date{ }
\maketitle  

\begin{abstract}
Recovery of a sparse signal from a nonlinear system arises in many practical applications including compressive sensing, image reconstruction and machine learning.
In this paper, a fast block nonlinear Bregman-Kaczmarz method with averaging is presented for nonlinear sparse signal recovery problems. Theoretical analysis proves that the averaging block nonlinear Bregman-Kaczmarz method with both constant stepsizes and adaptive stepsizes are convergent. 
Numerical experiments demonstrate the effectiveness of the averaging block nonlinear Bregman-Kaczmarz method, which converges faster than the existing nonlinear Bregman-Kaczmarz methods. 
\end{abstract}


\section{Introduction}
\label{sec:intro_abnbk}	  
Consider the sparse solution of the nonlinear equations
\begin{equation}\label{eqn:nonlineareq}
 F(x) = 0,
 \end{equation}
 where $x\in\mathbb{R}^n$ and $F(x)=(F_1(x),F_2(x),\cdots,F_m(x))^T$ is a nonlinear differentiable function. This nonlinear problem often arises in many scientific and engineering applications, for instance, compressed sensing \cite{2013Blu,2024LMR}, signal processing \cite{2017HWJ}, image reconstruction \cite{2017KMW}, wireless communication \cite{2020GZLYS} and deep learning \cite{2022LZB}.
 To find a sparse representation $\hat{x}$ satisfying $F(\hat{x})=0$, \eqref{eqn:nonlineareq} can be reformulated as the nonlinear constrained optimization problem
\begin{equation}
 \label{eqn:NBKproblem}
 \min\limits_{x\in \mathbb{R}^n} \varphi(x), \quad \text{s.t.} F(x) = 0,
\end{equation}
where $\varphi(x)$ is a convex and nonsmooth function which is also called sparsity inducing function.

The nonlinear Kaczmarz method \cite{2022YLG} is a simple and effective iterative method for solving the nonlinear system \eqref{eqn:NBKproblem}.
Let $F^{\prime}(x)=[\nabla F_1(x), \cdots, \nabla F_m(x)]^T$ be the Jacobian matrix of $F$ at $x$ and $\nabla F_i(x)^T$ be its $i$th row. At each iteration of the nonlinear Kaczmarz method \cite{2022WLB}, the current iterate $x_k$ is orthogonally projected onto the solution set of the local linearization of a component of equation $F_{i_k}(x)=0$ at $x_k$, that is, $x_{k+1}$ is obtained by solving the following constrained optimization problem  
\begin{equation*}
x_{k+1}=\mathop{\text{argmin}}\limits_{x\in \mathbb{R}^{n}}\|x-x_k\|_2^2,\quad \text{s.t.} \quad F_{i_k}(x_k)+\nabla F_{i_k}(x_k)^T(x-x_k)=0, 
\end{equation*}
where the index $i_k$ is cyclically or randomly selected from $[m]=\{1,\cdots,m\}$. 
The nonlinear Kaczmarz method was further accelerated by selecting the working row corresponding to the maximum residual of partial \cite{2023ZBLW,2024ZLT} or complete \cite{2023ZWZ,2024XYb} nonlinear equations. 
However, the nonlinear Kaczmarz method is efficient to compute the smooth solutions but difficult to capture special features of the solutions such as sparsity and piecewise constancy \cite{2013JW,2013AGKF}.
 
Recently, in order to reconstruct the nonsmooth solution of \eqref{eqn:NBKproblem}, by combining the nonlinear Kaczmarz method with Bregman projections, a nonlinear Bregman-Kaczmarz method was proposed and its expected convergence was proposed \cite{2024GLW}. At each iteration of the nonlinear Bregman-Kaczmarz method, the current iterate $x_k$ is obtained by solving the following constrained optimization problem
\begin{equation}\label{eq:bregnonlinearopti}
x_{k+1}= \mathop{\mathrm{argmin}}_{x\in\mathbb{R}^n} D_\varphi^{x_k^*}(x_k,x),\quad \text{s.t.} \quad F_{i_k}(x_k)+\nabla F_{i_k}(x_k)^T(x-x_k)=0,  
\end{equation}
where $D_\varphi(\cdot)$ represents the Bregman distance with respect to a convex function $\varphi$ and the index $i_k$ is randomly selected. 
By choosing $\varphi$ to be a sparsity inducing function, the nonlinear Bregman-Kaczmarz can be used for sparse signal recovery from nonlinear measurements \cite{2013BE}. 
For more studies on the nonlinear Bregman-Kaczmarz method, we refer the readers to \cite{2021GHT,2023TL,2023GC}. 
 
In this paper, in order to accelerate the convergence of the nonlinear Bregman-Kaczmarz method, an averaging block nonlinear Bregman-Kaczmarz method is developed for nonlinear sparse signal recovery.
Theoretical analysis gives the upper bound of the convergence rate of the averaging block nonlinear Bregman-Kaczmarz method with both constant stepsizes and adaptive stepsizes. 
Numerical experiments on different nonlinear sparse signal recovery problems show that the averaging block nonlinear Bregman-Kaczmarz method is effective and the remarkable acceleration in convergence to the solution.

The rest of this paper is organized as follows. 
In Section \ref{secabnbrgk_abnbk}, a fast block nonlinear Bregman-Kaczmarz method with the averaging technique is developed.
The convergence theory of the averaging block nonlinear Bregman-Kaczmarz method with constant stepsizes and adaptive stepsizes is established in Section \ref{sec:convanal}. 
Numerical experiments on nonlinear sparse signal recovery are provided to show the effectiveness of the proposed method in Section \ref{secnumer_abnbk}.
Finally, some conclusions are drawn in Section \ref{secconclu_abnbk}.

\section{Fast block nonlinear Bregman-Kaczmarz method} \label{secabnbrgk_abnbk}

In this section, the nonlinear Bregman-Kaczmarz method is firstly reviewed and then a fast block nonlinear Bregman-Kaczmarz method with averaging is presented for solving the nonlinear problem \eqref{eq:bregnonlinearopti}. 
Moreover, the convergence property of the averaging block nonlinear Kaczmarz method is analyzed on Bregman projections.

Let $\varphi:\mathbb{R}^n \to (-\infty,+\infty]$ be a convex function and its effective domain be 
$\text{dom }\varphi = \{x\in\mathbb{R}^n: \varphi(x)<\infty\} \neq\emptyset. $
Assuming that $\varphi$ is lower semicontinuous with
$\varphi(x) \leq \lim\inf\limits_{y\to x} \varphi(y)$  for all $x\in\mathbb{R}^n$, and is supercoercive, that is
$ \lim\limits_{\|x\|_2\to\infty} \frac{\varphi(x)}{\|x\|_2} = +\infty. $

\begin{definition}\rm
The subdifferential of $\varphi$ at $x\in\text{dom }\varphi$ is defined as 
$$
\partial\varphi(x)=\big\{x^*\in\mathbb{R}^n: \varphi(y) \geq \varphi(x) + \langle x^*,y-x\rangle, \quad y\in \text{dom }\varphi\big\},
$$
where $x^*\in\partial\varphi(x)$ represents a subgradient of $\varphi$ at $x$.
\end{definition}

\begin{definition}\rm
The convex conjugate of $\varphi$ is defined as
$$
\varphi^*(x^*) = \sup_{x\in\mathbb{R}^n} \langle x^*,x\rangle - \varphi(x), \qquad x^*\in\mathbb{R}^n, 
$$
where the function $\varphi^*$ is convex and lower semicontinuous.
\end{definition}

\begin{definition}\label{def:bregdist} \rm
The Bregman distance \cite{2019FD} between $x$ and $y$ with respect to $\varphi$ and a subgradient $x^*\in\partial \varphi(x)$ is defined as
$$
D_\varphi^{x^*}(x,y) = \varphi(y) - \varphi(x) - \langle x^*,y-x\rangle, \qquad x,y\in\text{dom }\varphi.
$$
\end{definition}
Using Fenchel's equality $\varphi^*(x^*)=\langle x^*,x\rangle - \varphi(x)$ for $x^*\in\partial\varphi(x)$,
the Bregman distance with the conjugate function can be rewritten as
 $
D_\varphi^{x^*}(x,y) = \varphi^*(x^*)-\langle x^*,y\rangle + \varphi(y).
 $
If $\varphi$ is differentiable at $x$, then the subdifferential $\partial\varphi(x)$ contains the single element $\nabla\varphi(x)$ and it yields that
$
 D_\varphi(x,y) = D^{\nabla\varphi(x)}_\varphi(x,y) = \varphi(y) - \varphi(x) - \langle \nabla\varphi(x),y-x\rangle. 
$
\begin{definition}\rm \label{def:eucbregman}
The function $\varphi$ is called $\sigma$-{strongly} convex with respect to a norm $\|\cdot\|_2$ for some $\sigma>0$, if for all $x,y\in\text{dom } \partial\varphi$ it satisfies that $\frac{\sigma}{2}\|x-y\|_2^2\leq D_\varphi^{x^*}(x,y)$. 
\end{definition}

\begin{definition}\rm\textbf{({Bregman projection} \cite{2019FD}).} \label{defbrgproj}
 Let $C \subset \mathbb{R}^n$ be a nonempty convex set, $C \cap \text{dom }\varphi\neq\emptyset$, $x\in\text{dom }\partial\varphi$ and $x^*\in\partial\varphi(x)$. The Bregman projection of $x$ onto $C$ with respect to $\varphi$ and $x^*$ is the point $\Pi_C^{x^*}(x)\in C \cap \text{dom }\varphi$ such that 
	\begin{equation*}
			 \Pi_C^{x^*}(x)= \arg\min_{y\in C} D_\varphi^{x^*}(x,y).
	\end{equation*}
\end{definition} 

Considering the Bregman projections onto hyperplanes 
	$$ H(\gamma,\beta):= \{x\in\mathbb{R}^n: \langle \gamma,x\rangle = \beta\}, \qquad \gamma\in\mathbb{R}^n, \ \beta\in\mathbb{R}, $$
	and halfspaces 
	$$ H^{\leq(\geq)}(\gamma,\beta):= \{x\in\mathbb{R}^n: \langle \gamma,x\rangle \leq(\geq) \beta\}, \qquad \gamma\in\mathbb{R}^n, \ \beta\in\mathbb{R}.$$

Given an appropriate convex function $\varphi\colon\mathbb{R}^n\to \mathbb{R} \cup \{+\infty\}$ with 
$\overline{\text{dom }\partial\varphi}= C$.
The nonlinear Bregman-Kaczmarz method \cite{2024GLW} was proposed for solving \eqref{eqn:NBKproblem} by calculating the Bregman projection with respect to $\varphi$ onto the solution set of  the local linearization of a component of equation $F_{i_k}(x)=0$ at the current iterate $x_k$, where the index $i_k$ is uniform randomly chosen from $i_k\in 
\{1,\ldots, n\}$. That is to obtain $x_{k+1}$ by solving the constrained optimization problem
\begin{equation} \label{eq:ourmethodintro}
	x_{k+1} =  \mathop{\mathrm{argmin}}_{x\in\mathbb{R}^n} D_\varphi^{x_k^*}(x_k,x), \qquad \text{s.t. } x\in H_k,
	\end{equation}
with
\begin{equation*}
H_k:= \{x\in\mathbb{R}^n: F_{i_k}(x_k) + \langle \nabla F_{i_k}(x_k), x-x_k\rangle = 0\}= H(\nabla F_{i_k}(x_k), \beta_k),
\end{equation*} 
where
\begin{equation}\label{eqn:beta_k}
    \beta_k = \langle \nabla F_{i_k}(x_k),x_k\rangle - F_{i_k}(x_k).
 \end{equation} 
The convergence properties of the nonlinear Bregman-Kaczmarz method in a classical setting of nonlinearity are restated in Theorem \ref{thm:convtccNBK}.

\begin{theorem}\rm 
	\label{thm:convtccNBK}
Let \cite[Assumption 1]{2024GLW} hold true, $\varphi$ be $\sigma$-strongly convex and $M$-smooth. Let each function $F_i$ satisfies the local tangential cone condition with some $\eta\in (0,\frac{\sigma}{2M})$, $\hat x\in S$ and $x_0\in B_{r,\varphi}(\hat x)$. Moreover, assume that the Jacobian $F'(x)$ has full column rank for all $x\in B_{r,\varphi}(\hat x)$ and $p_{\min}=\min\limits_{i=1,...,n} p_i>0$. 
Then, the iterates $x_k$ generated by the nonlinear Bregman-Kaczmarz method fulfill that 
   \begin{equation} \label{eqn:tcc_linear_Alg2}
	\frac{\sigma}{2}\mathbb{E}\big[\|x_k-\hat x\|_2^2] 
	   \leq \mathbb{E}\big[D_\varphi^{x_{k}^*}(x_{k},\hat x)\big] 
	\leq \Big( 1- \frac{ \sigma\big(\frac12-\eta\frac{M}{\sigma}\big)p_{\min}}{(1+\eta)^2\kappa_{\min}^2} \Big)^k \mathbb{E}\big[D_\varphi^{x_0^*}(x_0,\hat x)\big],
  \end{equation} 
where $\kappa_{\min}= \min\limits_{x\in B_{r,\varphi}(\hat x)} \min\limits_{\|y\|_2=1} \frac{\|F'(x)\|_F}{ \|F'(x)y\|_2}$ and $\|\cdot\|_F$ is the Frobenius norm.	
\end{theorem} 
Note that the averaging acceleration strategies \cite{2001CGG} are widely used in iterative methods, such as stochastic gradient descent \cite{2013SZ,2018LZZL}, coordinate descent methods \cite{2013N,2023SD} and Kaczmarz methods \cite{2019N,2020DSS,2023XYZ}.

In this work, to accelerate the convergence of the nonlinear Bregman-Kaczmarz method,  the averaging block technique are introduced.
Specifically, for each selected block index set $\mathcal{I}_k\subset\{1,\cdots,m\}$, the block nonlinear Bregman-Kaczmarz method iterates by
\begin{equation} \label{eqiterabnk}
\begin{aligned}
 x_{k+1}^*&= x_k^*- \alpha_k\left( \sum\limits_{i\in\mathcal{I}_k}\omega_k^{(i)} \frac{\sigma F_i(x_k)}{\left\|\nabla F_i(x_k) \right\| ^2_2} \nabla F_i(x_k) \right), \\
x_{k+1}& = \nabla \varphi^*(x_{k+1}^*),
\end{aligned}
\end{equation} 
where the weights satisfy $\omega_k^{(i)}\in[0,1]$ such that $\sum\limits_{i\in \mathcal{I}_k}\omega_k^{(i)} =1$ and $\alpha_k\in(0,2)$ is the stepsize.
Then, a class of averaging block nonlinear Bregman-Kaczmarz methods is developed in Algorithm \ref{alg:ABNBK}.
 
\begin{algorithm}[!htbp]
  \caption{Averaging block nonlinear Bregman-Kaczmarz method}
 \label{alg:ABNBK}
		\begin{algorithmic}[1]
		\Require $x_{0}^*\in \mathbb{R}^n, x_0 = \nabla \varphi^*(x_0^*), r_0= -F(x_0)$, $\sigma>0$ and $\theta\in(0,1]$. 
        \Ensure $x_\ell$.
			\For{$k=0,1,\cdots,\ell-1$}
			 \State Select the index set 
   $\mathcal{I}_k\subset\{1,\cdots,m\}$.
			 \State Update $x_{k+1}^*= x_k^*- \alpha_k\left( \sum\limits_{i\in\mathcal{I}_k}\omega_k^{(i)}\frac{\sigma F_i(x_k)}{\left\|\nabla F_i(x_k) \right\| ^2_2} \nabla F_i(x_k)  \right)$.
         	\State update $x_{k+1} = \nabla \varphi^*(x_{k+1}^*)$.
     \State Compute residual $r_{k+1}= -F(x_{k+1})$.
			\EndFor
		\end{algorithmic} 
\end{algorithm}	
Note that when $\theta=1$ and $\alpha_k=1$, Algorithm \ref{alg:ABNBK} method covers the maximum residual nonlinear Bregman-Kaczmarz method \cite{2023ZWZ}.
Note that when the function $\varphi(x) = \frac12\|x\|_2^2$,  Algorithm \ref{alg:ABNBK} method covers the averaging block nonlinear Kaczmarz method \cite{2024XYb}.

In Algorithm \ref{alg:ABNBK}, a simple choice for the weights is
$\omega_k^{(i)}=  {\left\|\nabla F_i(x_k) \right\|_2^2 }/\sum_{i\in\mathcal{I}_k} {\left\|\nabla F_i(x_k) \right\|_2^2}\in(0,1)$ for all $k\geq 0$ and then we can obtain the following update 
\begin{equation}\label{eqspupdate} 
  x_{k+1}^*= x_k^* - \alpha_k\frac{\sigma(F'_{\mathcal{I}_k}(x_k))^T F_{\mathcal{I}_k}(x_k)}{\left\|F'_{\mathcal{I}_k}(x_k) \right\|^2_2} \text{ and } x_{k+1} = \nabla \varphi^*(x_{k+1}^*).
\end{equation} 
For the stepsize $\alpha_k$, we mainly focus on two choices:
\begin{itemize}
  \item constant stepsize, that is, for any $k>0$, $\alpha_k$ is equal to a constant $\alpha>0$.
 \item adaptive stepsize, which is based on the idea of extrapolation
\begin{equation}\label{eq:abnbkadpsize}
\alpha_k= \delta \frac{\sum_{i\in\mathcal{I}_k}\hat{\omega}_k^{(i)}(F_i(x_k))^2}{\left\|\sum_{i\in\mathcal{I}_k}\hat{\omega}_k^{(i)}F_i(x_k)(\nabla F_i(x_k))^T \right\|_2^2},\quad k\geq 0,
\end{equation}
where $\hat{\omega}_k^{(i)}=\omega_k^{(i)}/\|\nabla F_i(x_k) \|_2^2$ and $\delta\in[1,2)$.
This extrapolated stepsize is also introduced for convex feasibility problems\cite{2019NRP} and for linear systems of equations\cite{2019N}.
\end{itemize}
Moreover, we consider the set $\mathcal{I}_k$ to be determined by the following greedy selection rule
\begin{equation}\label{eqabnkset}
 \mathcal{I}_k=\{ i_k| |r_k^{(i_k)}|^2\geq \theta\max\limits_{1\leq i\leq m} |r_k^{(i)}|^2 \}, \quad\theta\in(0,1].
\end{equation}
For more related work on greedy strategies, we refer the reader to \cite{2018BW,2020NZ,2023XYZ,2024YY}.

\section{ Convergence analysis }\label{sec:convanal}
First, some lemmas are introduced, which is crucial for the following convergence analysis of the averaging block nonlinear Bregman-Kaczmarz method.

 \begin{lemma}\rm
	\label{lem:normwithbregman}
		If $\varphi\colon\mathbb{R}^n\to\mathbb{R}$ is proper, convex and lower semicontinuous, then the following statements are equivalent
		\begin{enumerate}[(i)]
			\item $\varphi$ is $\sigma$-strongly convex with respect to $\|\cdot\|_2$. 
			\item $ \langle x^*-y^*, x-y\rangle \geq \sigma\|x-y\|_2^2$,  for all $x,y\in\mathbb{R}^n$ and $x^*\in\partial\varphi(x)$, $y^*\in\partial\varphi(y)$.
			\item The function $\varphi^*$ is $\tfrac{1}{\sigma}$-smooth with respect to $\|\cdot\|_2$.
		\end{enumerate}
\end{lemma}
  
\begin{lemma}{\em \cite{2024ZLT}}\label{lemftau}
\rm If the function $F$ satisfies the local tangential cone condition, then for any $x_1, x_2\in \mathbb{R}^n$ and an index subset $\mathcal{I}\subset \left\{1,2,\cdots, m \right\}$, it holds that
\begin{equation}\label{eq:lemftaun1}
\left\| F_{\mathcal{I}}(x_1)- F_{\mathcal{I}}(x_2)- F'_{\mathcal{I}}(x_1)(x_1-x_2) \right\|^2_2 \leq \eta^2  \left\|F_{\mathcal{I}}(x_1)- F_{\mathcal{I}}(x_2)\right\|^2_2 
\end{equation}
and
\begin{equation}\label{eq:lemftaun2}
\left\|F_{\mathcal{I}}(x_1)- F_{\mathcal{I}}(x_2)  \right\|_2^2\geq \frac{1}{(1+\eta)^2}\left\| F'_{\mathcal{I}}(x_1)(x_1- x_2) \right\|_2^2.
\end{equation}
\end{lemma}

\begin{lemma}\label{lem:matsingular}
\rm Let $A\in \mathbb{R}^{m\times n}$ be any nonzero real matrix. For every vector $u\in  {\rm range}(A)$,  
\begin{equation*}
  \sigma_{\min}^2(A) \left\|u \right\|_2^2\leq \left\|A^T u \right\|^2_2 \leq \sigma_{\max}^2(A) \left\|u \right\|_2^2,
\end{equation*}
where ${\rm range}(A)$, $\sigma_{\min}(A)$  and $\sigma_{\max}(A)$ are the column space, the nonzero minimum and maximum singular values of $A$, respectively.
\end{lemma}

 \begin{lemma}\label{lem:bregrecursion}
 \rm  If $\varphi$ is a $\sigma$-strongly convex function, from the definition of Bregman distance, it yields the following recursion
    \begin{equation}\label{eqrecurbregiter}
    D_\varphi^{x_{k+1}^*}(x_{k+1},\hat{x})\leq
    D_\varphi^{x_{k}^*}(x_{k},\hat{x})+ \langle x_{k+1}^*-x_{k}^*,x_k-\hat{x}\rangle + \frac{1}{2\sigma} \|x_{k+1}^*-x_{k}^*\|^2_2.
\end{equation}
\begin{proof}
From the Definition \eqref{def:bregdist} and $x_{k}= \nabla \varphi^*(x_{k}^*)$, it follows that 
\begin{equation*}
\begin{aligned}
   & D_\varphi^{x_{k+1}^*}(x_{k+1},\hat{x}) = \varphi^*(x_{k+1}^*) - \langle x_{k+1}^*,\hat{x}\rangle + \varphi(\hat{x})\\
   & \leq \varphi^*(x_{k}^*) + \langle \nabla\varphi^*(x_{k}^*),x_{k+1}^*-x_{k}^*\rangle + \frac{1}{2\sigma} \|x_{k+1}^*-x_{k}^*\|^2_2 - \langle x_{k+1}^*,\hat{x}\rangle + \varphi(\hat{x})\\
    &\leq \varphi^*(x_{k}^*)- \langle x_{k}^*,\hat{x}\rangle + \varphi(\hat{x})
    + \langle x_{k+1}^*-x_{k}^*,x_k-\hat{x}\rangle + \frac{1}{2\sigma} \|x_{k+1}^*-x_{k}^*\|^2_2 \\
   & \leq D_\varphi^{x_{k}^*}(x_{k},\hat{x})+ \langle x_{k+1}^*-x_{k}^*,x_k-\hat{x}\rangle + \frac{1}{2\sigma} \|x_{k+1}^*-x_{k}^*\|^2_2. 
    \end{aligned}    
\end{equation*}
\end{proof}
\end{lemma}
Lemma \ref{lem:bregrecursion} gives a elementary recursion about $D_\varphi^{x_{k+1}^*}(x_{k+1},\hat{x})$, which is useful for proving the convergence rate of the averaging block nonlinear Kaczmarz method.
 
\subsection{Averaging block nonlinear Bregman-Kaczmarz methods with constant stepsize}
In this subsection, the averaging block nonlinear Kaczmarz method with constant stepsize $\alpha_k=\alpha$ and weights $\omega_k^{(i)}= {\left\|\nabla F_i(x_k) \right\|_2^2 }/\sum_{i\in\mathcal{I}_k} {\left\|\nabla F_i(x_k) \right\|_2^2}$ is considered and its convergence analysis is given. 
In this case, the update formula of the averaging block nonlinear Kaczmarz method in Algorithm \ref{alg:ABNBK} becomes 
\begin{equation}\label{eq:iterconNBK}
\begin{aligned}
x^*_{k+1}&= x^*_{k}- \alpha\frac{(F'_{\mathcal{I}_k}(x_k))^T F_{\mathcal{I}_k}(x_k) }{\left\|F'_{\mathcal{I}_k}(x_k) \right\|^2_2 }, \\
x_{k+1}& = \nabla \varphi^*(x_{k+1}^*). 
\end{aligned}
\end{equation}

\begin{lemma}\label{lemcstabnbkiter}
 \rm Let $\varphi$ be $\sigma$-strongly convex function. If the function $F$ satisfies the local tangential cone condition, $\eta=\max\limits_i \eta_i< \frac{1}{2}$ and there exists a sparse vector $\hat{x}$ satisfies $F(\hat{x})=0$,  then from the iteration formula \eqref{eq:iterconNBK}, $\alpha\in(0,2(1-\eta))$ and $\mathcal{I}_k\subset \left\{1,2,\cdots, m \right\} $, it satisfies that
\begin{equation}
 D_\varphi^{x_{k+1}^*}(x_{k+1},\hat{x})\leq
    D_\varphi^{x_{k}^*}(x_{k},\hat{x})-\frac{(2(1-\eta)\alpha-\alpha^2)\sigma}{2}\frac{\left\|F_{\mathcal{I}_k}(x_k)\right\|_2^2}{\left\| F'_{\mathcal{I}_k}(x_k)\right\|_2^2}.   
\end{equation}
\end{lemma}

\begin{proof}
Inserting the iterative formula of ABNBK method with constant stepsize into inequality \eqref{eqrecurbregiter}, we have
	\begin{equation*}
	\begin{aligned}
	&D_\varphi^{x_{k+1}^*}(x_{k+1},\hat{x})\leq
	D_\varphi^{x_{k}^*}(x_{k},\hat{x})+ \langle x_{k+1}^*-x_{k}^*,x_k-\hat{x}\rangle + \frac{1}{2\sigma} \|x_{k+1}^*-x_{k}^*\|^2_2\\
	&\leq D_\varphi^{x_{k}^*}(x_{k},\hat{x})
	+  \left \langle -\alpha \frac{\sigma(F'_{\mathcal{I}_k}(x_k))^TF_{\mathcal{I}_k}(x_k)}{\left\|F'_{\mathcal{I}_k}(x_k)\right\|^2_2}, x_{k}-\hat{x} \right \rangle  
	+\frac{1}{2\sigma}\left\|\alpha\frac{\sigma(F'_{\mathcal{I}_k}(x_k))^T F_{\mathcal{I}_k}(x_k)}{\left\|F'_{\mathcal{I}_k}(x_k) \right\|^2_2 }\right\|^2_2 \\
	&\leq D_\varphi^{x_{k}^*}(x_{k},\hat{x})
	-\alpha \frac{ \sigma F^T_{\mathcal{I}_k}(x_k)F'_{\mathcal{I}_k}(x_k)}{\left\|F'_{\mathcal{I}_k}(x_k)\right\|_2^2}(x_k-\hat{x})
	+ \frac{\alpha^2\sigma}{2}\left\|\frac{(F'_{\mathcal{I}_k}(x_k))^TF_{\mathcal{I}_k}(x_k)}{\left\|F'_{\mathcal{I}_k}(x_k)\right\|_2^2} \right\|_2^2\\
	&\leq D_\varphi^{x_{k}^*}(x_{k},\hat{x})
	+\frac{\sigma}{2}\left(\alpha^2\left\|\frac{(F'_{\mathcal{I}_k}(x_k))^TF_{\mathcal{I}_k}(x_k)}{\left\|F'_{\mathcal{I}_k}(x_k)\right\|_2^2} \right\|_2^2 -2\alpha\frac{  F^T_{\mathcal{I}_k}(x_k)F'_{\mathcal{I}_k}(x_k)}{\left\|F'_{\mathcal{I}_k}(x_k)\right\|_2^2}(x_k-\hat{x}) \right)\\
	&\leq D_\varphi^{x_{k}^*}(x_{k},\hat{x})
	+\frac{\sigma}{2}\left( \frac{(\alpha^2-2\alpha)\left\|F_{\mathcal{I}_k}(x_k)\right\|_2^2}{\left\| F'_{\mathcal{I}_k}(x_k) \right\|_2^2} \right.\\
	&\quad \left. +\frac{2\alpha\cdot F^T_{\mathcal{I}_k}(x_k)}{\left\| F'_{\mathcal{I}_k}(x_k) \right\|_2^2}\left( F_{\mathcal{I}_k}(x_k)-F_{\mathcal{I}_k}(\hat{x})-F'_{\mathcal{I}_k}(x_k)(x_k-\hat{x}) \right)\right). 
	\end{aligned}    
	\end{equation*}
Moreover, using the Cauchy-Schwarz inequality, it obtains that
\begin{equation*}
\begin{aligned}
 D_\varphi^{x_{k+1}^*}(x_{k+1},\hat{x})\leq &D_\varphi^{x_{k}^*}(x_{k},\hat{x}) 
 +\frac{\sigma}{2}\left((\alpha^2-2\alpha)\frac{\left\|F_{\mathcal{I}_k}(x_k)\right\|_2^2}{\left\| F'_{\mathcal{I}_k}(x_k)\right\|_2^2} \right.\\
&\left. + 2\alpha\frac{\left\|F^T_{\mathcal{I}_k}(x_k)\right\|_2}{\left\| F'_{\mathcal{I}_k}(x_k) \right\|_2^2}\left\|F_{\mathcal{I}_k}(x_k)-F_{\mathcal{I}_k}(\hat{x})-F'_{\mathcal{I}_k}(x_k)(x_k-\hat{x})\right\|_2\right) .
\end{aligned}
\end{equation*}
From \eqref{eq:lemftaun1} in Lemma \ref{lemftau}, it yields that
\begin{equation}
\begin{aligned}
 D_\varphi^{x_{k+1}^*}(x_{k+1},\hat{x}) 
 &	\leq D_\varphi^{x_{k}^*}(x_{k},\hat{x}) 
	+\frac{\sigma}{2}\left((\alpha^2-2\alpha)\frac{\left\|F_{\mathcal{I}_k}(x_k)\right\|_2^2}{\left\| F'_{\mathcal{I}_k}(x_k)\right\|_2^2}
	+ 2\alpha\eta \frac{\left\|F^T_{\mathcal{I}_k}(x_k)\right\|_2}{\left\| F'_{\mathcal{I}_k}(x_k) \right\|_2^2}\left\|F_{\mathcal{I}_k}(x_k)-F_{\mathcal{I}_k}(\hat{x}) \right\|_2 \right)\\
&\leq D_\varphi^{x_{k}^*}(x_{k},\hat{x}) 
	+\frac{\sigma}{2}\left( (\alpha^2-2\alpha)\frac{\left\|F_{\mathcal{I}_k}(x_k)\right\|_2^2}{\left\| F'_{\mathcal{I}_k}(x_k)\right\|_2^2}
	+ 2\alpha\eta \frac{\left\|F^T_{\mathcal{I}_k}(x_k)\right\|_2}{\left\| F'_{\mathcal{I}_k}(x_k) \right\|_2^2}\left\|F_{\mathcal{I}_k}(x_k)\right\|_2 \right)\\
&\leq D_\varphi^{x_{k}^*}(x_{k},\hat{x}) 
	-\frac{(2(1-\eta)\alpha-\alpha^2)\sigma}{2}\frac{\left\|F_{\mathcal{I}_k}(x_k)\right\|_2^2}{\left\| F'_{\mathcal{I}_k}(x_k)\right\|_2^2},
\end{aligned}
\end{equation}
where the second inequality depends on $F(\hat{x})=0$. 
\end{proof}

Lemma \ref{lemcstabnbkiter} provides an important property for the sequence $\{x_k\}_{k=0}^{\infty}$ generated by the scheme \eqref{eq:iterconNBK} at each iteration. Next, we give an upper bound of the convergence rate for the averaging block nonlinear Bregman-Kaczmarz method with constant stepsize, which is described in Theorem \ref{thmcstABNBK}.

\begin{theorem}\label{thmcstABNBK}
\rm Under the same condition as Lemma \ref{lemcstabnbkiter}, for every $i\in\left\{1,\cdots, m \right\}$, if the nonlinear function $F_i$ satisfies the local tangential cone condition, $\eta= \max\limits_i \eta_i<\frac{1}{2}$ and exists a vector $\hat{x}$ such that $F(\hat{x})=0$, then the averaging block nonlinear Bregman-Kaczmarz method with constant stepsize $\alpha\in[1, 2(1-\eta))$ is convergent and 
\begin{equation}
   D_\varphi^{x_{k+1}^*}(x_{k+1},\hat{x}) 
 \leq \left(1- \frac{(2(1-\eta)\alpha-\alpha^2)\sigma}{M(1+\eta^2)\kappa_F^2(F'_{\mathcal{I}_k}(x_k))} \right)D_\varphi^{x_{k}^*}(x_{k},\hat{x}). 
\end{equation}
\end{theorem}
\begin{Proof}
From Lemma \ref{lemcstabnbkiter}, it follows that
\begin{equation*}
\begin{aligned}
	D_\varphi^{x_{k+1}^*}(x_{k+1},\hat{x}) 
	&\leq D_\varphi^{x_{k}^*}(x_{k},\hat{x})
	- \frac{(2(1-\eta)\alpha-\alpha^2)\sigma}{2}\frac{\left\|F_{\mathcal{I}_k}(x_k)-F_{\mathcal{I}_k}(\hat{x})\right\|_2^2}{\left\| F'_{\mathcal{I}_k}(x_k)\right\|_2^2}\\
	&\leq D_\varphi^{x_{k}^*}(x_{k},\hat{x})- \frac{(2(1-\eta)\alpha-\alpha^2)\sigma}{2(1+\eta)^2\left\| F'_{\mathcal{I}_k}(x_k)\right\|_2^2}\left\| F'_{\mathcal{I}_k}(x_k)(x_k- \hat{x}) \right\|_2^2\\
	& \leq D_\varphi^{x_{k}^*}(x_{k},\hat{x})- \frac{(2(1-\eta)\alpha-\alpha^2)\sigma}{2(1+\eta)^2\left\| F'_{\mathcal{I}_k}(x_k)\right\|_2^2} \sigma_{\min}^2(F'_{\mathcal{I}_k}(x_k))\left\|x_{k}-\hat{x} \right\|_2^2.
\end{aligned} 
\end{equation*}
The second and last inequality hold because \eqref{eq:lemftaun2} and Lemma \ref{lem:matsingular}, respectively. 
From Lemma \ref{lem:normwithbregman}(ii) and $\varphi$ is $M$-smooth, it concludes that
\begin{equation*}
\begin{aligned}
	D_\varphi^{x_{k+1}^*}(x_{k+1},\hat{x}) 
	&\leq D_\varphi^{x_{k}^*}(x_{k},\hat{x})-\frac{(2(1-\eta)\alpha-\alpha^2)\sigma}{M(1+\eta)^2\left\| F'_{\mathcal{I}_k}(x_k)\right\|_2^2} \sigma_{\min}^2(F'_{\mathcal{I}_k}(x_k)) D_\varphi^{x_k^*}(x_k,\hat x)  \\
	& \leq \left(1- \frac{(2(1-\eta)\alpha-\alpha^2)\sigma}{M(1+\eta)^2\kappa_F^2(F'_{\mathcal{I}_k}(x_k))} \right)D_\varphi^{x_{k}^*}(x_{k},\hat{x}).
\end{aligned} 
\end{equation*}
\end{Proof}

\subsection{Averaging block nonlinear Bregman-Kaczmarz methods with adaptive stepsize}
In this subsection, the convergence analysis of the averaging block nonlinear Bregman-Kaczmarz method with adaptive stepsize is given.
\begin{lemma}\label{lemadpabnbkiter}
 \rm  If function $\varphi$ is $\sigma$-strongly convex, $F$ satisfies the local tangential cone condition, $\eta=\max\limits_i \eta_i< \frac{1}{2}$ and a vector $\hat{x}$ satisfies $F(\hat{x})=0$, then from the averaging block iteration formula \eqref{eqiterabnk} with adaptive stepsize 
\eqref{eq:abnbkadpsize}, $\delta\in (0,2(1-\eta))$ and $\mathcal{I}_k\subset \{1,2,\cdots,m\}$, it obeys that
\begin{equation}\label{eqestiadp}
D_\varphi^{x_{k+1}^*}(x_{k+1},\hat{x}) 
 \leq D_\varphi^{x_{k}^*}(x_{k},\hat{x})-\frac{(2(1-\eta)\delta-\delta^2)\sigma}{2} \frac{\left\|F_{\mathcal{I}_k}(x_k)\right\|_2^2}{\sigma_{\max}^2(F'_{\mathcal{I}_k}(x_k) )}.
    \end{equation}
    
\begin{Proof}
Inserting the iterative formula of the ABNBK method with adaptive stepsize \eqref{eq:abnbkadpsize} into inequality \eqref{eqrecurbregiter}, it holds that
\begin{equation*}
\begin{aligned}
& D_\varphi^{x_{k+1}^*}(x_{k+1},\hat{x}) \leq D_\varphi^{x_{k}^*}(x_{k},\hat{x})  
+ \left \langle -\alpha_k\frac{\sigma(F'_{\mathcal{I}_k}(x_k))^TF_{\mathcal{I}_k}(x_k)}{\left\|F'_{\mathcal{I}_k}(x_k)\right\|^2_2}, x_{k}-\hat{x} \right \rangle
+ \frac{1}{2\sigma}\left\|\alpha_k\frac{\sigma(F'_{\mathcal{I}_k}(x_k))^TF_{\mathcal{I}_k}(x_k)}{\left\|F'_{\mathcal{I}_k}(x_k) \right\|^2_2 }\right\|^2_2\\
&\leq D_\varphi^{x_{k}^*}(x_{k},\hat{x})  + \frac{\sigma\delta^2}{2}\left\| \frac{\sum_{i\in\mathcal{I}_k}\hat{\omega}_k^{(i)}(F_i(x_k))^2}{\left\|\sum_{i\in\mathcal{I}_k}\hat{\omega}_k^{(i)}F_i(x_k)(\nabla F_i(x_k))^T \right\|_2^2}\frac{(F'_{\mathcal{I}_k}(x_k))^TF_{\mathcal{I}_k}(x_k)}{\left\|F'_{\mathcal{I}_k}(x_k) \right\|^2_2 }\right\|^2_2 \\
&\quad - {\sigma\delta}\left \langle \frac{\sum_{i\in\mathcal{I}_k}\hat{\omega}_k^{(i)}(F_i(x_k))^2}{\left\|\sum_{i\in\mathcal{I}_k}\hat{\omega}_k^{(i)}F_i(x_k)(\nabla F_i(x_k))^T \right\|_2^2}\frac{(F'_{\mathcal{I}_k}(x_k))^TF_{\mathcal{I}_k}(x_k)}{\left\|F'_{\mathcal{I}_k}(x_k)\right\|^2_2},  x_{k}-\hat{x} \right \rangle \\
&\leq D_\varphi^{x_{k}^*}(x_{k},\hat{x})+ \frac{\sigma}{2}\left( \delta^2\left\|\frac{\left\|F_{\mathcal{I}_k}(x_k)\right\|_2^2(F'_{\mathcal{I}_k}(x_k))^TF_{\mathcal{I}_k}(x_k)}{\left\|(F'_{\mathcal{I}_k}(x_k) )^TF_{\mathcal{I}_k}(x_k)\right\|_2^2} \right\|_2^2
- 2\delta \frac{ \left\|F_{\mathcal{I}_k}(x_k)\right\|_2^2F^T_{\mathcal{I}_k}(x_k)F'_{\mathcal{I}_k}(x_k)}{\left\|(F'_{\mathcal{I}_k}(x_k) )^TF_{\mathcal{I}_k}(x_k)\right\|_2^2}(x_k-\hat{x}) \right) \\
&\leq D_\varphi^{x_{k}^*}(x_{k},\hat{x}) + \frac{\sigma}{2}\left( 2\delta\frac{\left\|F_{\mathcal{I}_k}(x_k)\right\|_2^2F^T_{\mathcal{I}_k}(x_k)}{\left\|(F'_{\mathcal{I}_k}(x_k) )^TF_{\mathcal{I}_k}(x_k)\right\|_2^2} \left( F_{\mathcal{I}_k}(x_k)-F_{\mathcal{I}_k}(\hat{x})-F'_{\mathcal{I}_k}(x_k)(x_k-\hat{x}) \right) \right.\\
 & \left. \quad - 2\delta\frac{\left\|F_{\mathcal{I}_k}(x_k)\right\|_2^2F^T_{\mathcal{I}_k}(x_k)}{\left\|(F'_{\mathcal{I}_k}(x_k) )^TF_{\mathcal{I}_k}(x_k)\right\|_2^2} F_{\mathcal{I}_k}(x_k)+  \delta^2\frac{\left\|F_{\mathcal{I}_k}(x_k)\right\|_2^4}{\left\|(F'_{\mathcal{I}_k}(x_k) )^TF_{\mathcal{I}_k}(x_k)\right\|_2^2} \right)\\
&\leq D_\varphi^{x_{k}^*}(x_{k},\hat{x}) + \frac{\sigma}{2} \left( 2\delta\frac{\left\|F_{\mathcal{I}_k}(x_k)\right\|_2^2F^T_{\mathcal{I}_k}(x_k)}{\left\|(F'_{\mathcal{I}_k}(x_k) )^TF_{\mathcal{I}_k}(x_k)\right\|_2^2}\left( F_{\mathcal{I}_k}(x_k)-F_{\mathcal{I}_k}(\hat{x})-F'_{\mathcal{I}_k}(x_k)(x_k-\hat{x}) \right) \right.\\
&\left. \quad + (\delta^2-2\delta)\frac{\left\|F_{\mathcal{I}_k}(x_k)\right\|_2^4}{\left\|(F'_{\mathcal{I}_k}(x_k) )^TF_{\mathcal{I}_k}(x_k)\right\|_2^2} \right).
\end{aligned}
\end{equation*}
Furthermore, using the Cauchy-Schwarz inequality, it holds that
\begin{equation*}
\begin{aligned}
D_\varphi^{x_{k+1}^*}(x_{k+1},\hat{x}) &\leq D_\varphi^{x_{k}^*}(x_{k},\hat{x})
  + \frac{\sigma}{2} \left(2\delta\frac{\left\|F_{\mathcal{I}_k}(x_k) \right\|_2^2\left\|F^T_{\mathcal{I}_k}(x_k)\right\|_2}{\left\|(F'_{\mathcal{I}_k}(x_k) )^TF_{\mathcal{I}_k}(x_k)\right\|_2^2}\left\|F_{\mathcal{I}_k}(x_k)-F_{\mathcal{I}_k}(\hat{x})-F'_{\mathcal{I}_k}(x_k)(x_k-\hat{x})\right\|_2  \right.\\
&\left. \quad + (\delta^2-2\delta)\frac{\left\|F_{\mathcal{I}_k}(x_k)\right\|_2^4}{\left\|(F'_{\mathcal{I}_k}(x_k) )^TF_{\mathcal{I}_k}(x_k)\right\|_2^2}  \right).
    \end{aligned}
\end{equation*}
From \eqref{eq:lemftaun1} in Lemma \ref{lemftau} and $F(\hat{x})=0$, it yields that
\begin{equation}
\begin{aligned}
D_\varphi^{x_{k+1}^*}(x_{k+1},\hat{x}) &\leq D_\varphi^{x_{k}^*}(x_{k},\hat{x})
  + \frac{\sigma}{2} \left( 2\delta\eta \frac{\left\|F_{\mathcal{I}_k}(x_k) \right\|_2^2}{\left\|(F'_{\mathcal{I}_k}(x_k) )^T F_{\mathcal{I}_k}(x_k)\right\|_2^2}\left\|F^T_{\mathcal{I}_k}(x_k)\right\|_2 \left\|F_{\mathcal{I}_k}(x_k) \right\|_2  \right.\\
&\left. \quad + (\delta^2-2\delta)\frac{\left\|F_{\mathcal{I}_k}(x_k)\right\|_2^4}{\left\|(F'_{\mathcal{I}_k}(x_k) )^TF_{\mathcal{I}_k}(x_k)\right\|_2^2}  \right). \\
 &\leq D_\varphi^{x_{k}^*}(x_{k},\hat{x})- \frac{\sigma\left(2(1-\eta)\delta-\delta^2\right)}{2} \frac{\left\|F_{\mathcal{I}_k}(x_k)\right\|_2^4}{\left\|(F'_{\mathcal{I}_k}(x_k) )^TF_{\mathcal{I}_k}(x_k)\right\|_2^2}.
\end{aligned}
\end{equation}
From Lemma \ref{lem:matsingular}, it has
\begin{equation}
    \begin{aligned}
  D_\varphi^{x_{k+1}^*}(x_{k+1},\hat{x}) &\leq D_\varphi^{x_{k}^*}(x_{k},\hat{x})
  - \frac{\sigma\left(2(1-\eta)\delta-\delta^2\right)}{2} \frac{\left\|F_{\mathcal{I}_k}(x_k)\right\|_2^4}{\sigma_{\max}^2(F'_{\mathcal{I}_k}(x_k) )\left\| F_{\mathcal{I}_k}(x_k)\right\|_2^2}\\
  &\leq D_\varphi^{x_{k}^*}(x_{k},\hat{x})-\frac{\sigma\left(2(1-\eta)\delta-\delta^2\right)}{2} \frac{\left\|F_{\mathcal{I}_k}(x_k)\right\|_2^2}{\sigma_{\max}^2(F'_{\mathcal{I}_k}(x_k) )},
    \end{aligned}
\end{equation}
which is exactly the estimate in \eqref{eqestiadp}.
\end{Proof}
\end{lemma}
Lemma \ref{lemadpabnbkiter} shows a crucial relationship between solution vector $\hat{x}$ and iterate sequence $\{x_k\}_{k=0}^{\infty}$ generated by Algorithm \ref{alg:ABNBK} with adaptive stepsize defined by \eqref{eq:abnbkadpsize}. 

The convergence of Algorithm \ref{alg:ABNBK} with adaptive stepsize is proved and its upper bound of the convergence rate given in Theorem \ref{thm:conv_adpMRABNBK}. 
  
\begin{theorem}\label{thm:conv_adpMRABNBK}
\rm Under the same condition as Theorem \ref{thm:convtccNBK}, for every $i\in\{1,\cdots,m\}$, if the nonlinear function $F_i$ satisfies the local tangential cone condition, $\eta= \max\limits_i \eta_i<\frac{1}{2}$ and exists a vector such that $F(x_{\ast})=0$, then the averaging block nonlinear Bregman-Kaczmarz method with adaptive stepsize and $\delta\in(0, 2(1-\eta))$ is convergent. Moreover, it satisfies that 
\begin{equation}
   \frac{2}{\sigma}\|x_k-\hat x\|_2^2 \leq D_\varphi^{x_{k+1}^*}(x_{k+1},\hat{x}) \leq \left(1- \frac{\sigma(2(1-\eta)\delta-\delta^2)}{M(1+\eta^2)\kappa^2(F'_{\mathcal{I}_k}(x_k)) } \right)D_\varphi^{x_{k}^*}(x_{k},\hat{x}).
\end{equation} 
\end{theorem}
\begin{Proof}
From Lemma \ref{lemadpabnbkiter} and $F(\hat{x})=0$, it follows that
\begin{equation}
 \begin{aligned}
 D_\varphi^{x_{k+1}^*}(x_{k+1},\hat{x}) &\leq D_\varphi^{x_{k}^*}(x_{k},\hat{x})
 -\frac{\sigma(\left(2(1-\eta)\delta-\delta^2\right))}{2} \frac{\left\|F_{\mathcal{I}_k}(x_k)-F_{\mathcal{I}_k}(\hat{x})\right\|_2^2}{\sigma_{\max}^2(F'_{\mathcal{I}_k}(x_k) )} \\
&\leq  D_\varphi^{x_{k}^*}(x_{k},\hat{x})-\frac{\sigma(\left(2(1-\eta)\delta-\delta^2\right))}{2(1+\eta^2)\sigma^2_{\max}(F'_{\mathcal{I}_k}(x_k)) }\left\|F'_{\mathcal{I}_k}(x_k)(x_k- x_{\ast}) \right\|_2^2\\
&\leq D_\varphi^{x_{k}^*}(x_{k},\hat{x})-\frac{\sigma(\left(2(1-\eta)\delta-\delta^2\right))}{2(1+\eta^2)\sigma^2_{\max}(F'_{\mathcal{I}_k}(x_k)) }\sigma_{\min}^2(F'_{\mathcal{I}_k}(x_k))\left\|x_{k}-x_{\ast} \right\|_2^2\\
&\leq D_\varphi^{x_{k}^*}(x_{k},\hat{x})-\frac{\sigma(\left(2(1-\eta)\delta-\delta^2\right))}{M(1+\eta^2)\sigma^2_{\max}(F'_{\mathcal{I}_k}(x_k)) }\sigma_{\min}^2(F'_{\mathcal{I}_k}(x_k))D_\varphi^{x_{k}^*}(x_{k},\hat{x})\\
& \leq \left(1- \frac{\sigma(\left(2(1-\eta)\delta-\delta^2\right))}{M(1+\eta^2)\kappa^2(F'_{\mathcal{I}_k}(x_k)) } \right)D_\varphi^{x_{k}^*}(x_{k},\hat{x}).
\end{aligned} 
\end{equation}
The second inequality holds because \eqref{eq:lemftaun2} and the last inequality depends on the definition of condition number $\kappa(\cdot)$. 
\end{Proof}

\section{Numerical experiments} \label{secnumer_abnbk}
In this section, numerical experiments are presented to verify the efficiency of the averaging block nonlinear Kaczmarz method for nonlinear sparse signal recovery.
The averaging block nonlinear Bregman-Kaczmarz method with constant stepsize (abbreviated as `ABNBK-c') and with adaptive stepsize (abbreviated as `ABNBK-a') are compared with the nonlinear Bregman-Kaczmarz method (abbreviated as `NBK') proposed in \cite{2024GLW} and the maximum residual nonlinear Bregman-Kaczmarz method (abbreviated as `MRNBK'). 

The optimal parameters for the averaging block nonlinear Bregman-Kaczmarz methods are experimentally selected by minimizing the number of iteration steps. For the nonlinear Bregman-Kaczmarz method and the maximum residual nonlinear Bregman-Kaczmarz method, the control index $i_k$ is chosen by the probability criterion $| r_k^{(i_k)}|^2/\|r_k \|_2^2$ and the greedy rule $i_k=\arg \max\limits_{1\leq i\leq m}|r_k(i_k)|^2$, respectively.

Consider the nonlinear sparse signal recovery \cite{2013LV}
\begin{equation}\label{eqmultiquadeq}
\min\limits_{x\in \mathbb{R}^n } \varphi(x), \quad \text{s.t. }F_i(x) = \frac12\langle x,A^{(i)}x\rangle + \langle b^{(i)},x\rangle + c^{(i)}=0, 
\end{equation}
where the data $A\in\mathbb{R}^{m\times n\times n}$, matrix $A^{(i)}\in\mathbb{R}^{n\times n}$, $b^{(i)}\in\mathbb{R}^n$, 
$c^{(i)}\in\mathbb{R}$, $i=1,\cdots,m$. 

In all experiments, the original sparse signal $\hat{x}\in\mathbb{R}^n$ is generated with its nonzero entries from the standard normal distribution. The sparsity of $\hat{x}$ is denoted `sp', that is the ratio of the number of nonzero elements and the total elements in $\hat{x}$.
Let 
$$
c^{(i)} = - \Big( \frac{1}{2}\langle \hat{x},A^{(i)}\hat{x}\rangle + \langle b^{(i)},\hat{x}\rangle \Big),
$$
where the data $b^{(i)}$ is generated with entries from the standard normal distribution.
For recovering the sparse signal $\hat x$, the sparse generating function $\varphi(x)=\lambda\|x\|_1+\frac12\|x\|_2^2$ and its corresponding soft shrinkage operator $S_{\lambda}(x)=\max\{ |x|-\lambda,0 \}\cdot \text{sign}(x) $($\lambda>0$) is used. 

In all experiments, let the sparse parameter be $\lambda=2$. All iterations initial from $x_0=\nabla\varphi^*(x_0^*)= S_\lambda(x_0^*)$ with $x_0^*$ is sampled from the standard normal distribution and terminate when the relative residual satisfies
$$
\frac{\left\|F(x_k)\right\|_2^2} {\left\|F(x_0) \right\|_2^2} \leq 10^{-6}
$$
or the number of iterations exceeds a maximum number, e.g., 1000.

\begin{example}\label{ex1:GAUmatrix}
In this experiment, each tested matrix $A^{(i)}(i=1,\cdots,m)$ is generated with its entries from standard normal distribution. 
\end{example}

In Tables \ref{tab:resultSP01GAU} and \ref{tab:resultSP005GAU}, the number of iteration steps and the elapsed CPU time for all methods when the size of $A$ and the value of $sp$ varies are listed, respectively. 

From Tables \ref{tab:resultSP01GAU} and \ref{tab:resultSP005GAU}, it is observed that
the ABNBK-a method requires the fewest iteration steps and least CPU time among all methods. Moreover, the nonlinear Bregman-Kaczmarz and the maximum residual Bregman-Kaczmarz methods are difficult to recover the signal as the size of the problem increases.
While the proposed ABNBK method can successfully reconstruct the sparse signal within the finite iteration steps, which indicate that the ABNBK method is advantageous for solving the large nonlinear problem.

\begin{table}[!htbp] 
\centering  
\caption{ Numerical results for Example \ref{ex1:GAUmatrix} with $sp=0.1$. } \label{tab:resultSP01GAU}  
  \resizebox{\textwidth}{!}{  
\begin{tabular}{ccccccccccc} 
\hline
\multirow{2}{*} {$m$} & \multirow{2}{*} {$n$}& \multirow{2}{*} {$sp$}   & \multicolumn{2}{c}{NBK} & \multicolumn{2}{c}{MRNBK} & \multicolumn{2}{c} {ABNBK-c} & \multicolumn{2}{c} {ABNBK-a} \\ 
 \cmidrule[0.25mm](lr){4-5} \cmidrule[0.25mm](lr){6-7} \cmidrule[0.25mm](lr){8-9} \cmidrule[0.25mm](lr){10-11}  & & & IT &CPU &IT &CPU &IT &CPU &IT &CPU  \\ 
\hline
 200 &	100 &	0.1&	 550&	 14.410&	 440&	 11.644&	 57&	 1.882&	 31&	 0.964\\  
 300 &	150 &	0.1&	 1000&	 69.937&	 973&	 68.079&	 98&	 8.012&	 61&	 4.800\\  
 400 &	200 &	0.1&	 1000&	 136.681&	 1000&	 145.929&	 889&	 143.792&	 319&	 52.946\\  
 500 &	250 &	0.1&	 1000&	 241.504&	 1000&	 248.077&	 77&	 20.577&	 72&	 20.346\\  
 600 &	300 &	0.1&	 1000&	 374.762&	 1000&	 395.748&	 71&	 29.028&	 59&	 26.273\\  
 800 &	350 &	0.1&	 1000&	 1372.571&	 1000&	 1376.741&	 572&	 898.566&	 295&	 463.534\\ 
 900 &	450 &	0.1&	 1000&	 1869.268&	 1000&	 1901.807&	 176&	 383.619&	 144&	 314.991\\  
 1000 &	500 &	0.1&	 1000&	 2497.861&	 1000&	 2558.757&	 351&	 1009.091&	 230&	 682.802\\  
\hline
\end{tabular}  
  }   
 \begin{tablenotes} 
 \footnotesize               
\item The parameters in the ABNK-c and ABNK-a methods are $(\alpha^*, \theta_2^*)=(1.9,0.1)$ and $(\omega^*, \theta_3^*)=(1.3,0.1)$, respectively. 
\end{tablenotes}
\end{table} 

\begin{table}[!htbp] 
\centering 
\caption{Numerical results for Example \ref{ex1:GAUmatrix} with $sp=0.05$.} \label{tab:resultSP005GAU}  
\begin{tabular}{ccccccccccc} 
\hline
\multirow{2}{*} {$m$} & \multirow{2}{*} {$n$}& \multirow{2}{*} {$sp$}   & \multicolumn{2}{c}{NBK} & \multicolumn{2}{c}{MRNBK} & \multicolumn{2}{c} {ABNBK-c} & \multicolumn{2}{c} {ABNBK-a} \\ 
 \cmidrule[0.25mm](lr){4-5} \cmidrule[0.25mm](lr){6-7} \cmidrule[0.25mm](lr){8-9} \cmidrule[0.25mm](lr){10-11}  & & & IT &CPU &IT &CPU &IT &CPU &IT &CPU  \\ 
\hline
 200 &	100 &	0.05&	 162&	 4.242&	 54&  1.413&  23&	0.750&	 17&  0.502\\  
 300 &	150 &	0.05&	 1000&	 66.172&   954&  64.939& 227&  17.180& 172&  13.950\\  
 400 &	200 &	0.05&	 686&	 93.031&	287&	39.925&	 28&   4.581&  17&	 2.740\\  
 500 &	250 &	0.05&	 1000&	 247.383&	 503&	 126.716&	 68&	 18.229&	 57&	 16.239\\  
 600 &	300 &	0.05&	 1000&	 380.897&	 954&	 360.259&	 99&	 40.723&	 83&	 33.302\\  
800 &	350 &	0.05&	 1000&	 1432.852&	 1000&	 1466.418&	 53&	 89.332&	 31&	 52.868\\   
 900 &	450 &	0.05&	 1000&	 1988.321&	 1000&	 1997.292&	 75&	 165.403&	 62&	 136.588\\  
 1000 &	500 &	0.05&	 1000&	 2579.332&	 1000&	 2611.398&	 56&	 162.497&	 35&	 99.414\\  
\hline
\end{tabular}  
 \begin{tablenotes} 
 \footnotesize               
\item The parameters in the ABNK-c and ABNK-a methods are $(\alpha^*, \theta_2^*)=(1.9,0.1)$ and $(\omega^*, \theta_3^*)=(1.3,0.1)$, respectively. 
\end{tablenotes}
\end{table} 

In Figures \ref{fig:RESvsITsp01GAU} and \ref{fig:RESvsITsp005GAU}, the curves of the norm of the relative residual versus the number of iteration steps and the solution error versus the number of iteration steps for all tested methods are plotted, respectively.

From Figures \ref{fig:RESvsITsp01GAU} and \ref{fig:RESvsITsp005GAU}, it is seen that the ABNBK-c and ABNBK-a method converge much faster than the NBK and MRNBK methods, which implies that the averaging technique is effective and can greatly improve the convergence speed of nonlinear Bregman-Kaczmarz method. 

\begin{figure}[!htbp] 
\centering 
\vspace{-0.4cm}  
 \subfigtopskip= 1pt  
\subfigbottomskip= 0.1pt  
\subfigcapskip= -5pt 
	\subfigure[$(m,n)=(200, 100)$]
	{
		\begin{minipage}[t]{0.48\linewidth}
			\centering
			\includegraphics[width=1\textwidth]{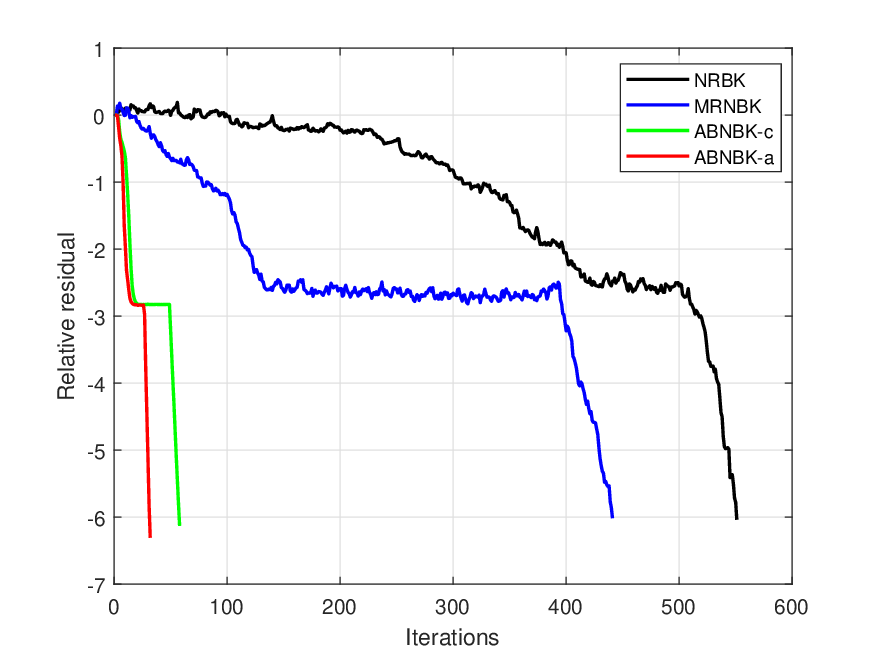}
		\end{minipage}
	}
	\subfigure[$(m,n)=(1000, 500)$]
	{
		\begin{minipage}[t]{0.48\linewidth}
			\centering 
			\includegraphics[width=1\textwidth]{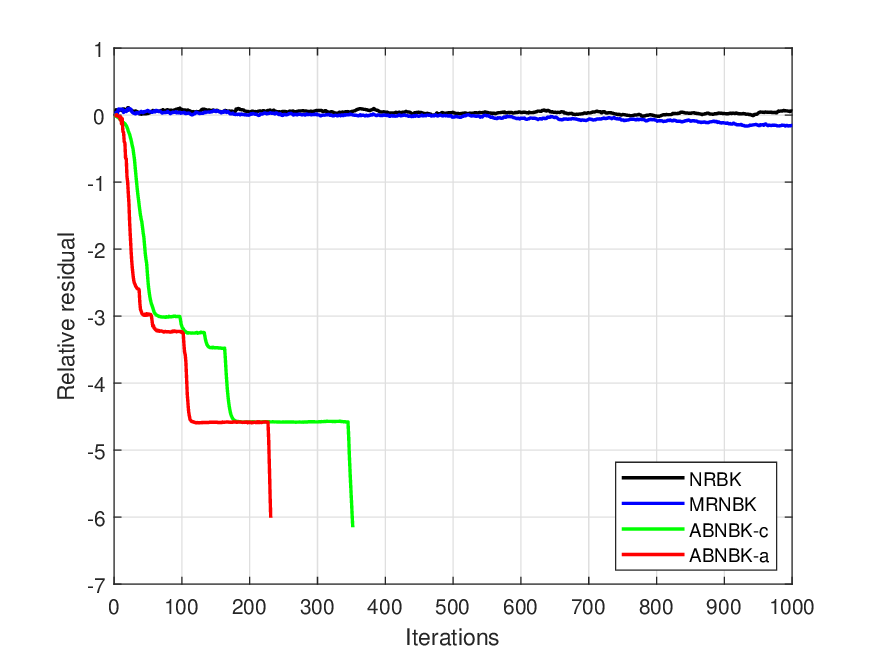}
		\end{minipage}
	} 
\caption{Convergence curves for Example \ref{ex1:GAUmatrix} with $sp=0.1$.} 
		\label{fig:RESvsITsp01GAU}
\end{figure}

\begin{figure}[!htbp] 
\centering 
\vspace{-0.4cm}  
 \subfigtopskip=1pt  
\subfigbottomskip=0.1pt  
\subfigcapskip=-5pt 
	\subfigure[$(m,n)=(200, 100)$]
	{
		\begin{minipage}[t]{0.48\linewidth}
			\centering
			\includegraphics[width=1\textwidth]{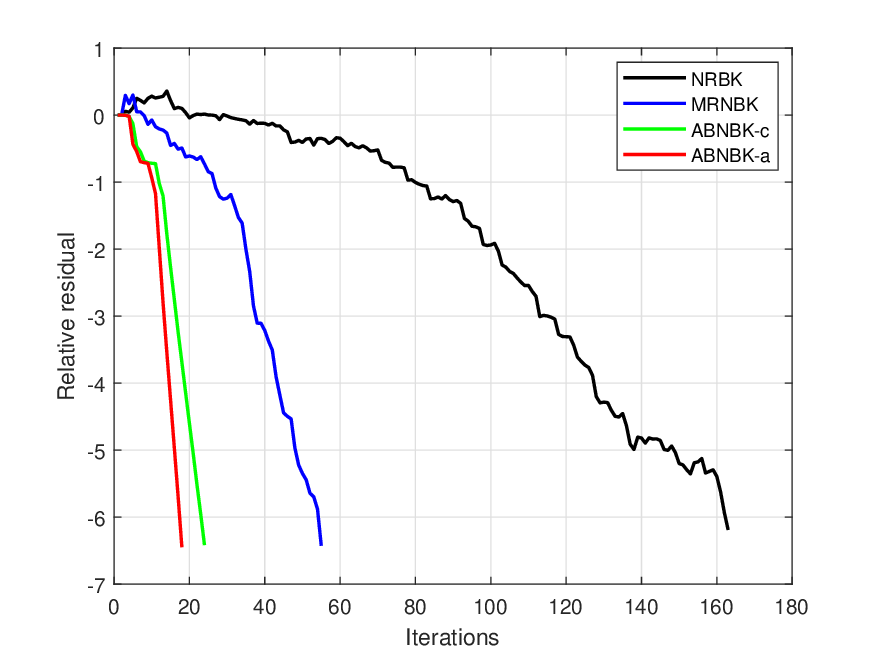}
		\end{minipage}
	}
	\subfigure[$(m,n)=(1000, 500)$]
	{
		\begin{minipage}[t]{0.48\linewidth}
			\centering 
			\includegraphics[width=1\textwidth]{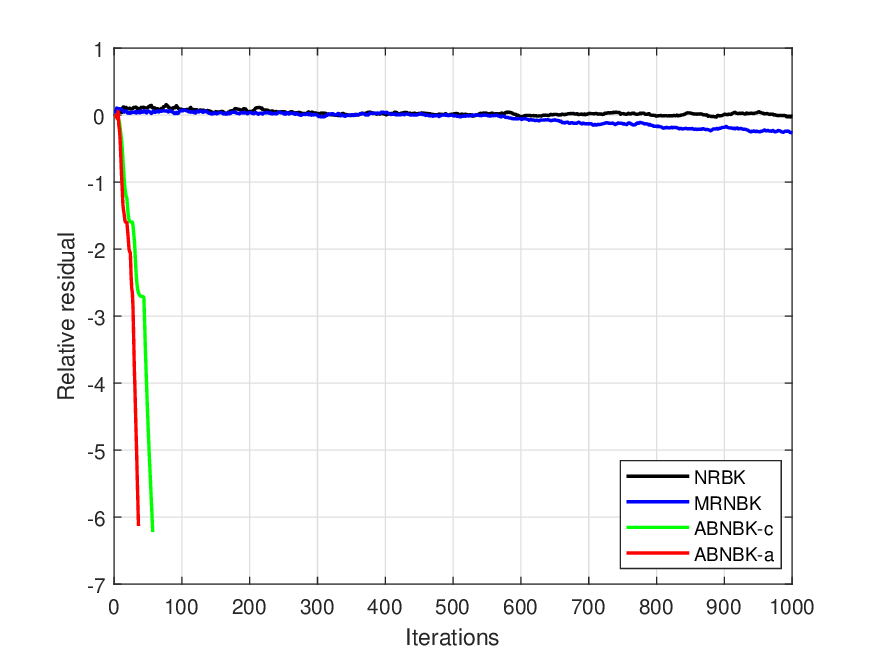}
		\end{minipage}
	} 
  \caption{Convergence curves for Example \ref{ex1:GAUmatrix} with $sp=0.05$.} 
		\label{fig:RESvsITsp005GAU}
\end{figure}

In Figures \ref{fig:xRecoversp01GAU} and \ref{fig:xRecoversp005GAU}, the original signal and the recovered signals by all methods when $n=500$, $sp=0.1$ and $sp=0.05$ are depicted, respectively.

From Figures \ref{fig:xRecoversp01GAU} and \ref{fig:xRecoversp005GAU}, it is observed that the recovered signals by the ABNBK-c and ABNBK-a method are more close to the original signal than other two methods.

\begin{figure}[!htbp] 
\centering 
\vspace{-0.4cm}  
 \subfigtopskip=1pt  
\subfigbottomskip=0.1pt  
\subfigcapskip=-5pt 
	\subfigure[NBK]
	{
		\begin{minipage}[t]{0.48\linewidth}
			\centering
			\includegraphics[width=1\textwidth]{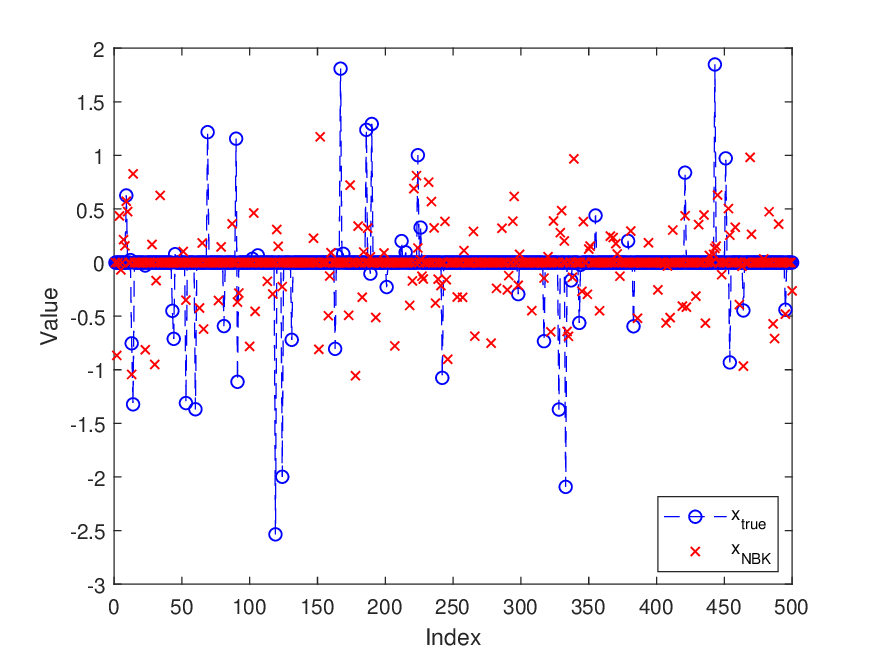}
		\end{minipage}
	}
	\subfigure[MRNBK]
	{
		\begin{minipage}[t]{0.48\linewidth}
			\centering 
			\includegraphics[width=1\textwidth]{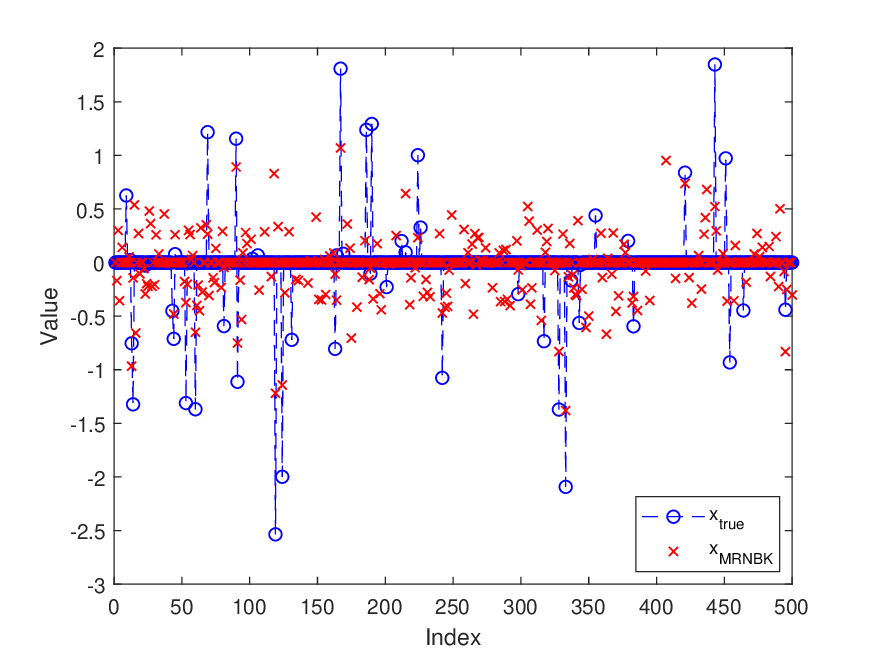}
		\end{minipage}
	} 
    \subfigure[ABNBK-c]
	{
		\begin{minipage}[t]{0.48\linewidth}
			\centering
			\includegraphics[width=1\textwidth]{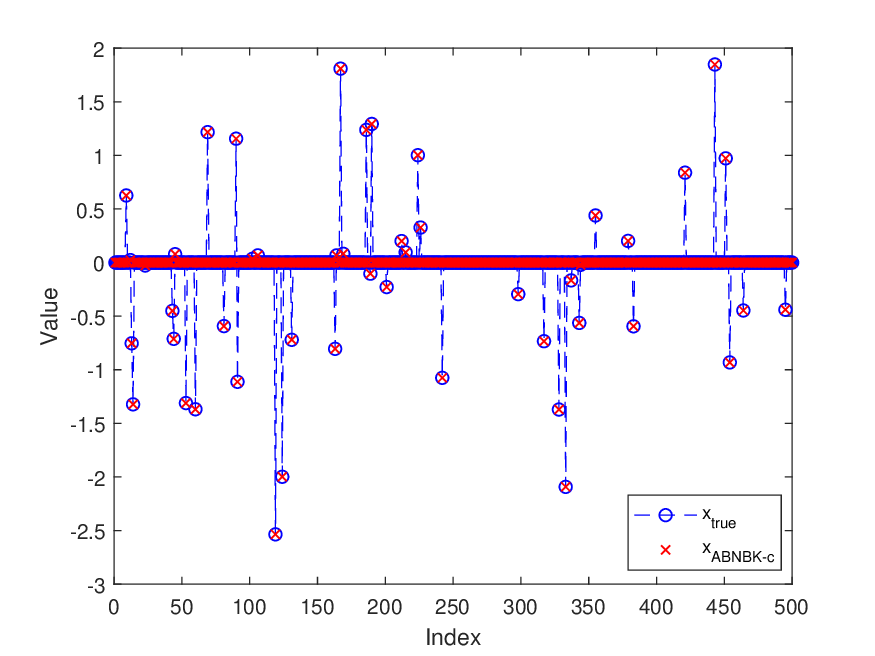}
		\end{minipage}
	}
	\subfigure[ABNBK-a]
	{
		\begin{minipage}[t]{0.48\linewidth}
			\centering 
			\includegraphics[width=1\textwidth]{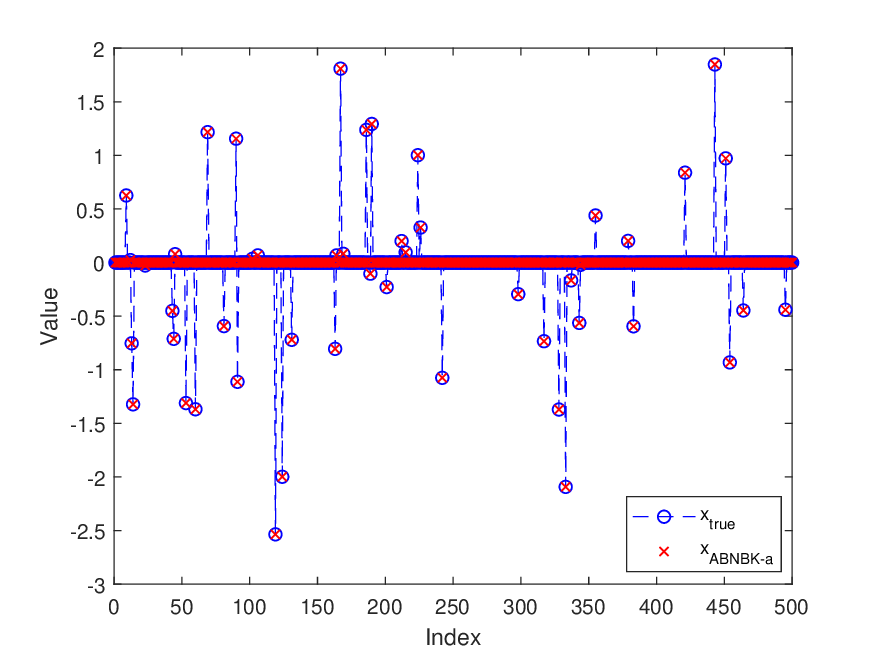}
		\end{minipage}
	} 
  \caption{The exact signal and recovered signals for Example \ref{ex1:GAUmatrix} with $sp=0.1$.} 
  \label{fig:xRecoversp01GAU}
\end{figure}

\begin{figure}[!htbp] 
\centering 
\vspace{-0.4cm}  
 \subfigtopskip=1pt  
\subfigbottomskip=0.1pt  
\subfigcapskip=-5pt 
	\subfigure[NBK]
	{
		\begin{minipage}[t]{0.48\linewidth}
			\centering
			\includegraphics[width=1\textwidth]{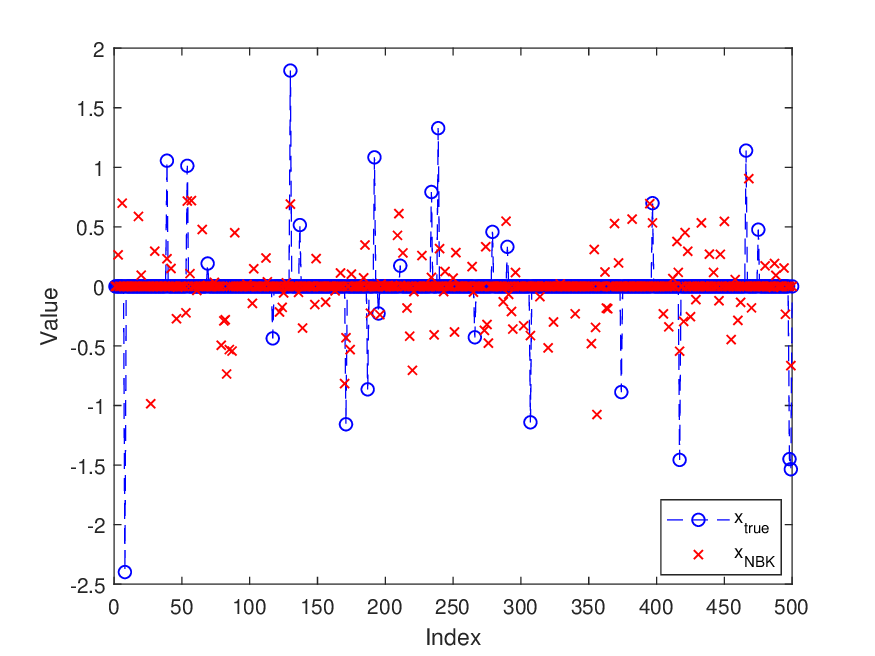}
		\end{minipage}
	}
	\subfigure[MRNBK]
	{
		\begin{minipage}[t]{0.48\linewidth}
			\centering 
			\includegraphics[width=1\textwidth]{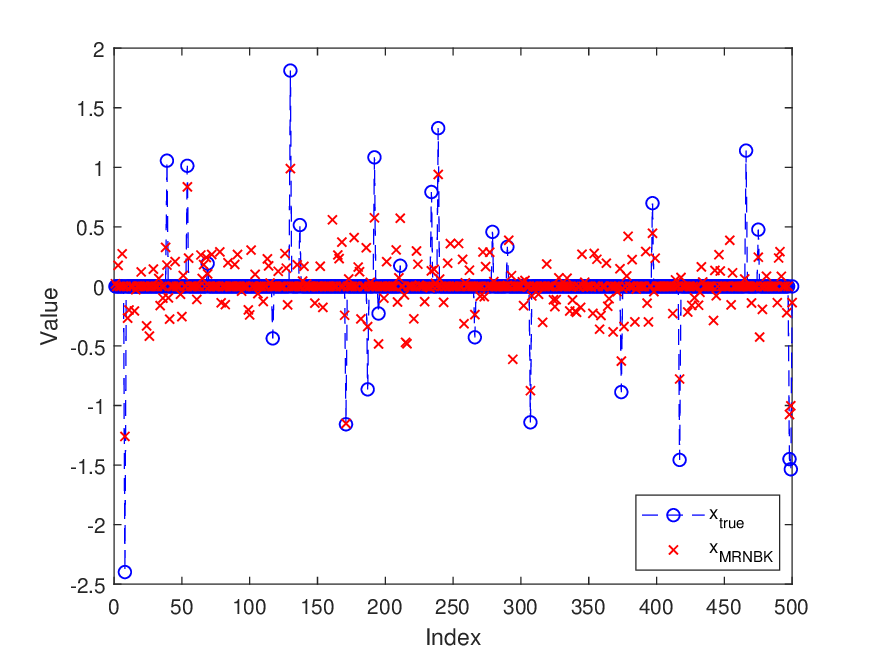}
		\end{minipage}
	} 
    \subfigure[ABNBK-c]
	{
		\begin{minipage}[t]{0.48\linewidth}
			\centering
			\includegraphics[width=1\textwidth]{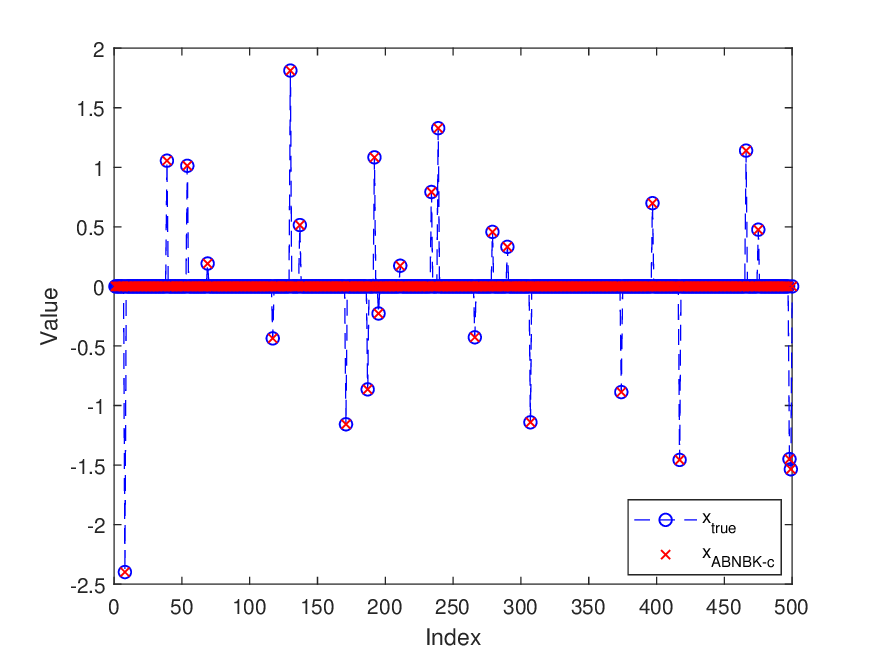}
		\end{minipage}
	}
	\subfigure[ABNBK-a]
	{
		\begin{minipage}[t]{0.48\linewidth}
			\centering 
			\includegraphics[width=1\textwidth]{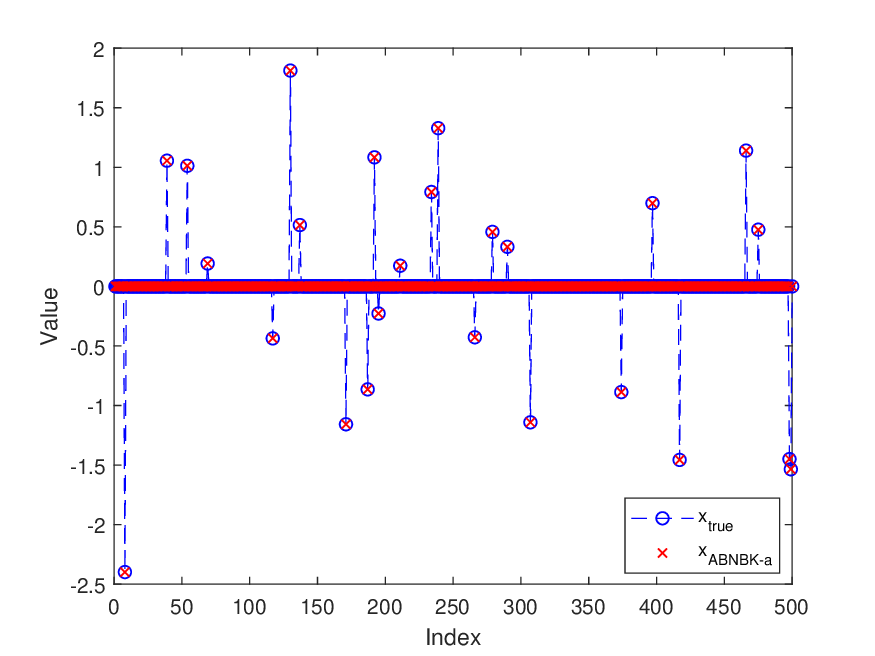}
		\end{minipage}
	} 
  \caption{The exact signal and recovered signals for Example \ref{ex1:GAUmatrix} with $sp=0.05$.} 
		\label{fig:xRecoversp005GAU}
\end{figure}

\begin{example}\label{ex1:DCTmatrix}
In this experiment, each tested matrix $A^{(i)}(i=1,\cdots,m)$ is a random partial discrete cosine transform (DCT) matrix which $j$-th column is generated through the expression
$$
A^{(i)}(:,j)= cos(2\pi (j-1)\xi), \quad j=1,\cdots,n,  
$$
where $\xi\in \mathbb{R}^{m\times 1}$ is a column vector with uniformly and independently sampled elements from $[0,1]$.
\end{example}

In Tables \ref{tab:resultSP01DCT} and \ref{tab:resultSP005DCT}, the number of iteration steps and the elapsed CPU time for all methods when the size of $A$ and the value of $sp$ varies are listed, respectively. 

From Tables \ref{tab:resultSP01DCT} and \ref{tab:resultSP005DCT}, it is observed that
the ABNBK-a method requires the fewest iteration steps and least CPU time among all methods. Moreover, the nonlinear Bregman-Kaczmarz and the maximum residual Bregman-Kaczmarz method are difficult to recover the signal as the size of the problem increases.
While the proposed ABNBK method can successfully reconstruct the sparse signal within the finite iteration steps, which indicates that the ABNBK methods are advantageous for solving large nonlinear problems.

\begin{table}[!htbp] 
\centering 
\caption{Numerical results for Example \ref{ex1:DCTmatrix} with $sp=0.1$.} \label{tab:resultSP01DCT} 
\resizebox{\textwidth}{!}{  
\begin{tabular}{ccccccccccc} 
\hline
\multirow{2}{*} {$m$} & \multirow{2}{*} {$n$}& \multirow{2}{*} {$sp$}   & \multicolumn{2}{c}{NBK} & \multicolumn{2}{c}{MRNBK} & \multicolumn{2}{c} {ABNBK-c} & \multicolumn{2}{c} {ABNBK-a} \\ 
 \cmidrule[0.25mm](lr){4-5} \cmidrule[0.25mm](lr){6-7} \cmidrule[0.25mm](lr){8-9} \cmidrule[0.25mm](lr){10-11}  & & & IT &CPU &IT &CPU &IT &CPU &IT &CPU  \\ 
\hline
 200 &	100 &	0.1&	 683&	 18.139&	 926&	 24.716&	 67&	 2.054&	 35&	 1.083\\  
 300 &	150 &	0.1&	 1000&	 64.848&	 1000&	 65.715&	 381&	 29.113&	 202&	 14.792\\  
 400 &	200 &	0.1&	 1000&	 139.383&	 1000&	 144.561&	 610&	 98.694&	 403&	 66.817\\  
 500 &	250 &	0.1&	 1000&	 249.151&	 1000&	 257.723&	 774&	 218.658&	 498&	 148.628\\  
 600 &	300 &	0.1&	 1000&	 415.380&	 1000&	 400.643&	 616&	 264.427&	 167&	 75.847\\  
 800 &	400 &	0.1&	 1000&	 1342.724&	 1000&	 1359.726&	 537&	 819.644&	 254&	 387.752\\  
 900 &	450 &	0.1&	 1000&	 2019.286&	 1000&	 3370.743&	 189&	 431.184&	 123&	 278.255\\  
 1000 &	500 &	0.1&	 1000&	 2555.362&	 1000&	 2618.649&	 137&	 399.672&	 82&	 246.498\\  
\hline
\end{tabular}  
 }   
 \end{table} 

\begin{table}[!htbp]  
\centering 
\caption{Numerical results for Example \ref{ex1:DCTmatrix} with $sp=0.05$.} \label{tab:resultSP005DCT} 
\resizebox{\textwidth}{!}{  
\begin{tabular}{ccccccccccc} 
\hline
\multirow{2}{*} {$m$} & \multirow{2}{*} {$n$}& \multirow{2}{*} {sp}   & \multicolumn{2}{c}{NBK} & \multicolumn{2}{c}{MRNBK} & \multicolumn{2}{c} {ABNBK-c} & \multicolumn{2}{c} {ABNBK-a} \\ 
 \cmidrule[0.25mm](lr){4-5} \cmidrule[0.25mm](lr){6-7} \cmidrule[0.25mm](lr){8-9} \cmidrule[0.25mm](lr){10-11}  & & & IT &CPU &IT &CPU &IT &CPU &IT &CPU  \\ 
\hline
 200 &	100 &	0.05&	 845&	 24.829&	 332&	 9.351&	 319&	 10.653&	 229&	 7.491\\  
 300 &	150 &	0.05&	 274&	 18.756&	 159&	 10.888&	 25&	 1.991&	 16&	 1.280\\  
 400 &	200 &	0.05&	 799&	 112.075&	 868&	 127.403&	 174&	 27.518&	 103&	 17.404\\  
 500 &	250 &	0.05&	 590&	 142.966&	 315&	 80.687&	 32&	 9.434&	 21&	 6.125\\  
 600 &	300 &	0.05&	 1000&	 409.482&	 1000&	 423.836&	 51&	 22.340&	 29&	 13.263\\  
 800 &	400 &	0.05&	 1000&	 1348.747&	 676&	 900.045&	 67&	 105.745&	 41&	 63.521\\  
 900 &	450 &	0.05&	 1000&	 1970.914&	 1000&	 1974.391&	 73&	 164.428&	 42&	 93.813\\  
 1000 &	500 &	0.05&	 1000&	 2422.198&	 1000&	 2409.087&	 55&	 148.657&	 27&	 72.452\\  
\hline
\end{tabular}  
 }   
 \end{table}

In Figures \ref{fig:RESvsITsp01DCT} and \ref{fig:RESvsITsp005DCT}, the curves of the norm of the relative residual versus the number of iteration steps and the solution error versus the number of iteration steps for all tested methods are shown, respectively.

From Figures \ref{fig:RESvsITsp01DCT} and \ref{fig:RESvsITsp005DCT}, it is seen that the ABNBK-c and ABNBK-a method converge much faster than the NBK and MRNBK methods, which shows the efficiency of the averaging technique. 

\begin{figure}[!htbp] 
\centering 
\vspace{-0.4cm}  
 \subfigtopskip=1pt  
\subfigbottomskip=0.1pt  
\subfigcapskip=-5pt 
	\subfigure[$(m,n)=(200, 100)$]
	{
		\begin{minipage}[t]{0.48\linewidth}
			\centering
			\includegraphics[width=1\textwidth]{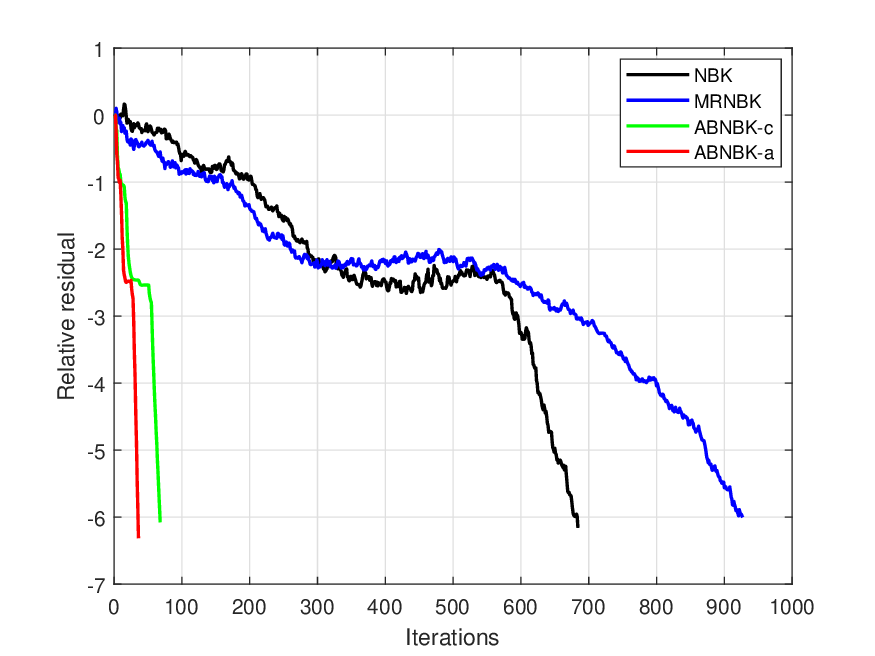}
		\end{minipage}
	}
	\subfigure[$(m,n)=(1000, 500)$]
	{
		\begin{minipage}[t]{0.48\linewidth}
			\centering 
			\includegraphics[width=1\textwidth]{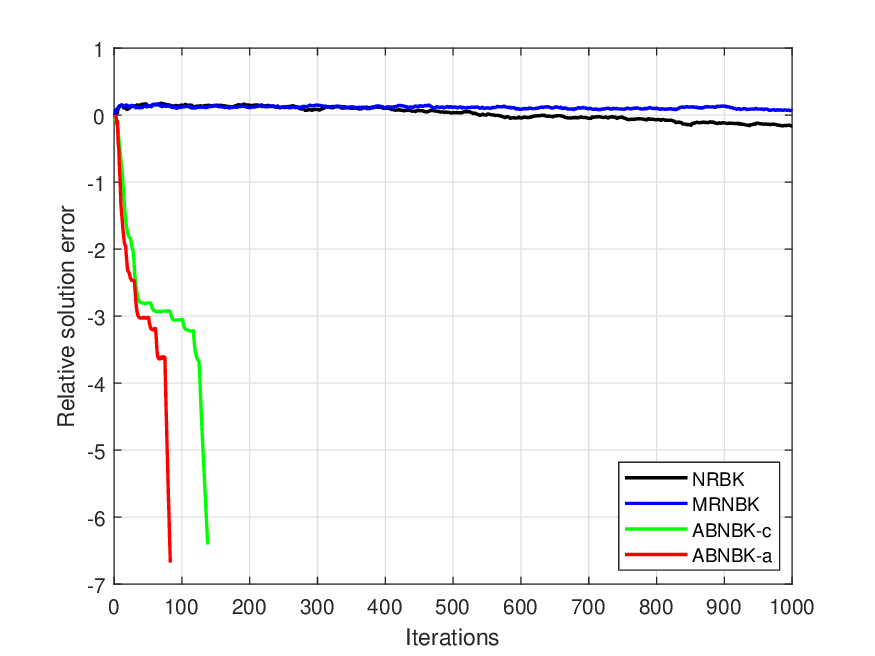}
		\end{minipage}
	} 
  \caption{Convergence curves for Example \ref{ex1:DCTmatrix} with $sp=0.1$.} 
		\label{fig:RESvsITsp01DCT}
\end{figure}

\begin{figure}[!htbp] 
\centering 
\vspace{-0.4cm}  
 \subfigtopskip=1pt  
\subfigbottomskip=0.1pt  
\subfigcapskip=-5pt 
	\subfigure[$(m,n)=(200, 100)$]
	{
		\begin{minipage}[t]{0.48\linewidth}
			\centering
			\includegraphics[width=1\textwidth]{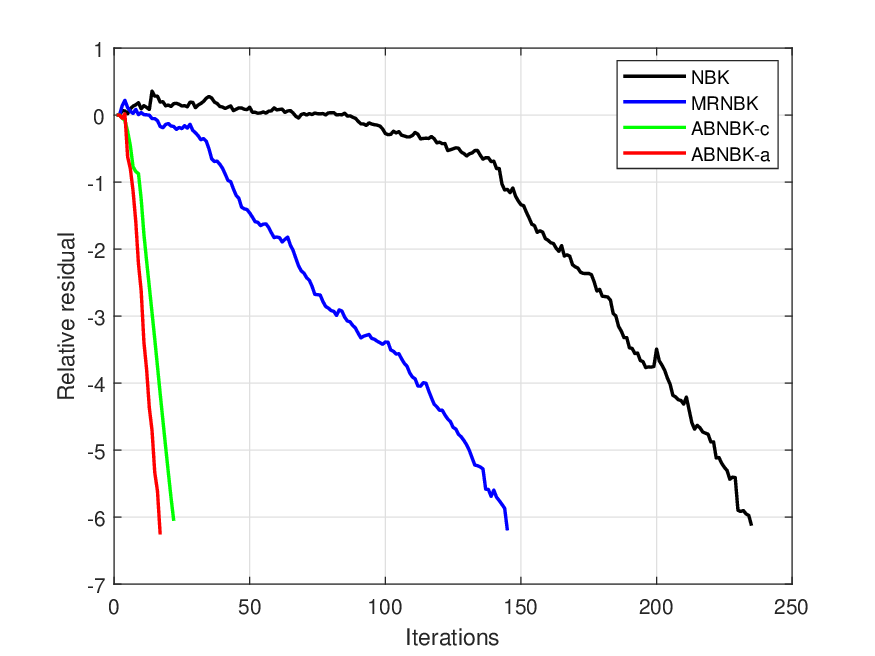}
		\end{minipage}
	}
	\subfigure[$(m,n)=(1000, 500)$]
	{
		\begin{minipage}[t]{0.48\linewidth}
			\centering 
			\includegraphics[width=1\textwidth]{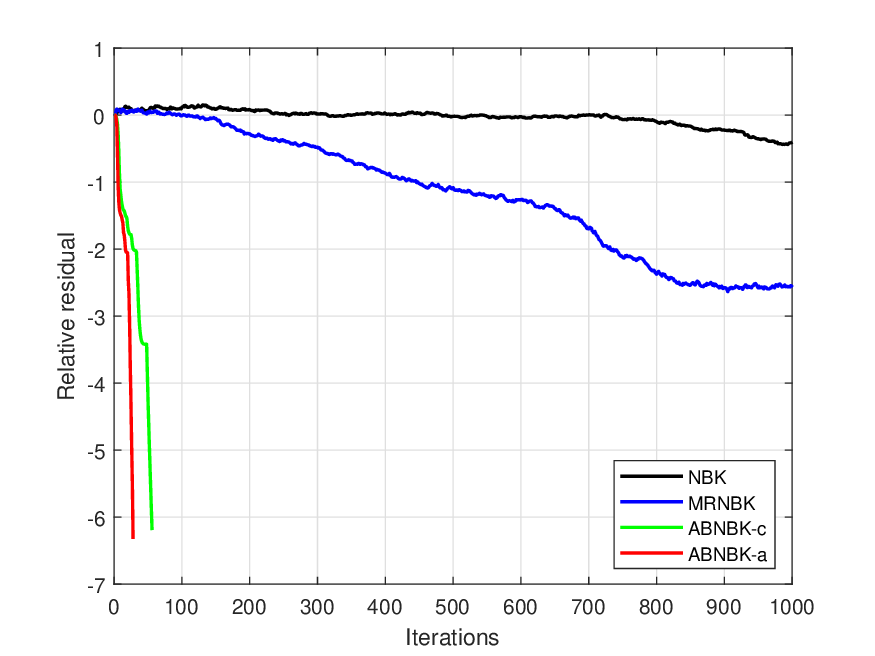}
		\end{minipage}
	} 
  \caption{Convergence curves for Example \ref{ex1:DCTmatrix} with $sp=0.05$.} 
		\label{fig:RESvsITsp005DCT}
\end{figure}

In Figures \ref{fig:xRecoversp01DCT} and \ref{fig:xRecoversp005DCT}, the original signal and the recovered signals by all methods when $n=500$, $sp=0.1$ and $sp=0.05$ are displayed, respectively.

From Figures \ref{fig:xRecoversp01DCT} and \ref{fig:xRecoversp005DCT}, it is observed that the recovered signals by the ABNBK-c and ABNBK-a method are more close to the original signal than other two methods.

\begin{figure}[!htbp] 
\centering 
\vspace{-0.4cm}  
 \subfigtopskip=1pt  
\subfigbottomskip=0.1pt  
\subfigcapskip=-5pt 
	\subfigure[NBK]
	{
		\begin{minipage}[t]{0.48\linewidth}
			\centering
			\includegraphics[width=1\textwidth]{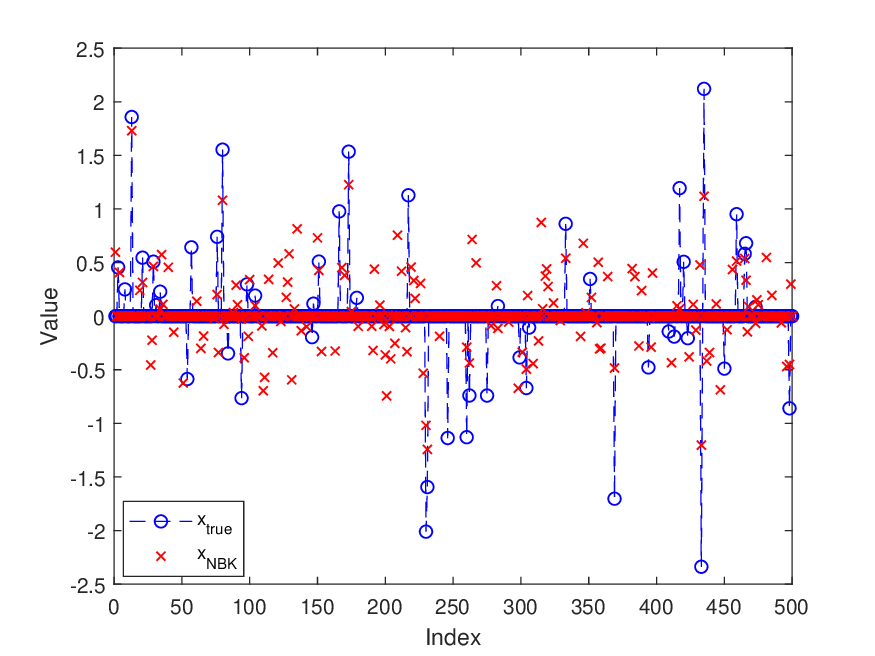}
		\end{minipage}
	}
	\subfigure[MRNBK]
	{
		\begin{minipage}[t]{0.48\linewidth}
			\centering 
			\includegraphics[width=1\textwidth]{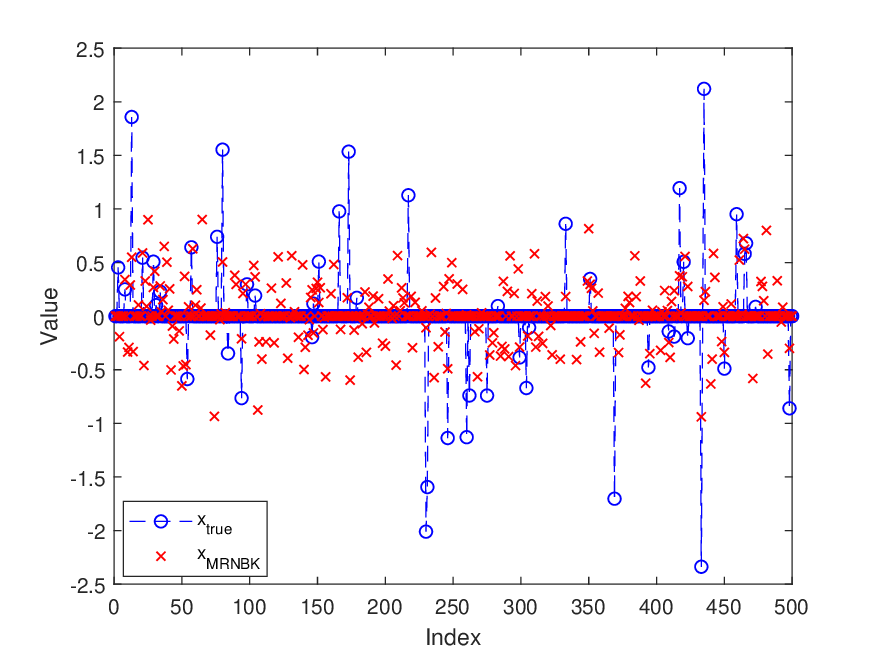}
		\end{minipage}
	} 
    \subfigure[ABNBK-c]
	{
		\begin{minipage}[t]{0.48\linewidth}
			\centering
			\includegraphics[width=1\textwidth]{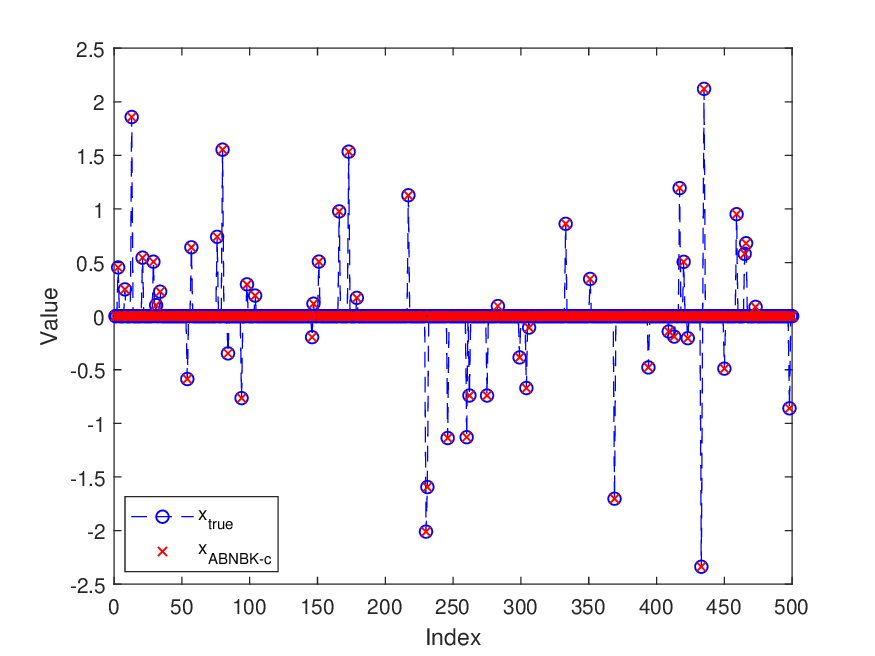}
		\end{minipage}
	}
	\subfigure[ABNBK-a]
	{
		\begin{minipage}[t]{0.48\linewidth}
			\centering 
			\includegraphics[width=1\textwidth]{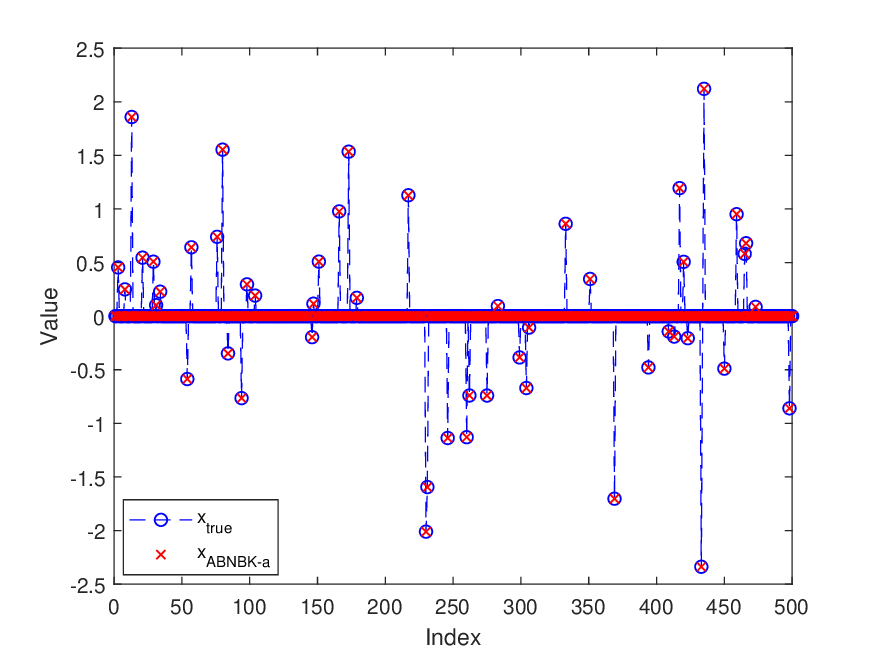}
		\end{minipage}
	} 
  \caption{The exact signal and recovered signals for Example \ref{ex1:DCTmatrix} with $sp=0.1$.} 
      \label{fig:xRecoversp01DCT}
\end{figure}

\begin{figure}[!htbp] 
\centering 
\vspace{-0.4cm}  
 \subfigtopskip=1pt  
\subfigbottomskip=0.1pt  
\subfigcapskip=-5pt 
	\subfigure[NBK]
	{
		\begin{minipage}[t]{0.48\linewidth}
			\centering
			\includegraphics[width=1\textwidth]{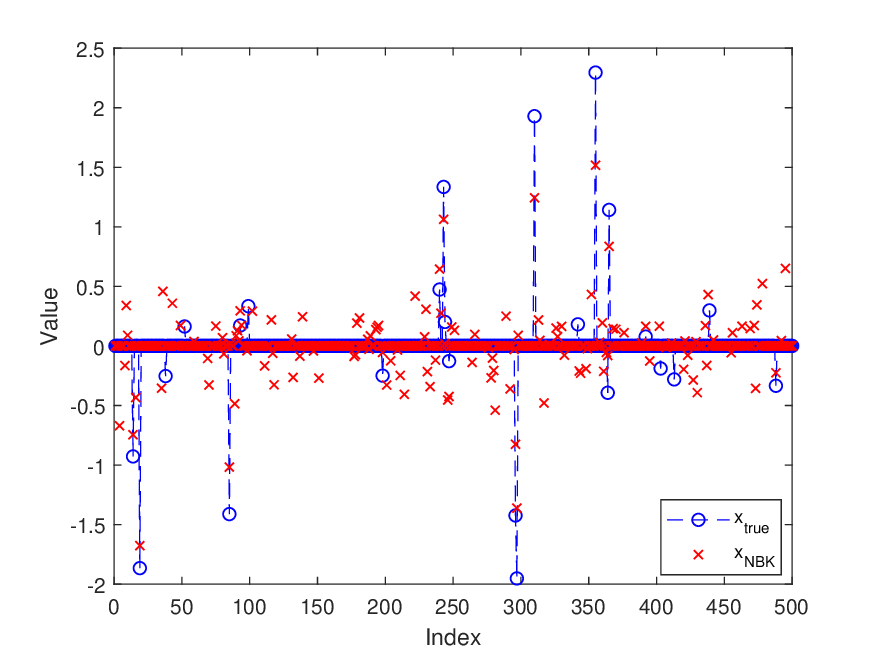}
		\end{minipage}
	}
	\subfigure[MRNBK]
	{
		\begin{minipage}[t]{0.48\linewidth}
			\centering 
			\includegraphics[width=1\textwidth]{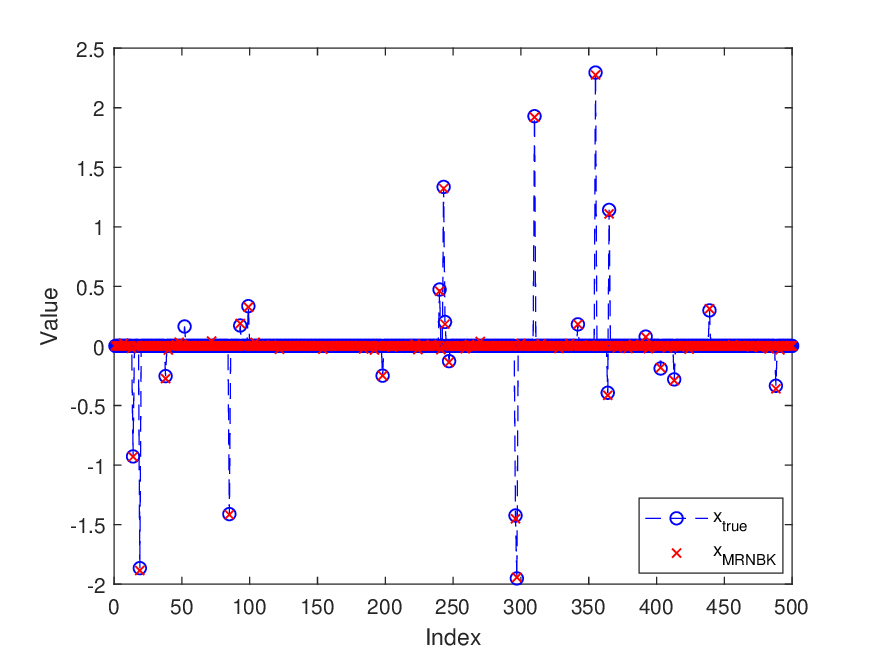}
		\end{minipage}
	} 
    \subfigure[ABNBK-c]
	{
		\begin{minipage}[t]{0.48\linewidth}
			\centering
			\includegraphics[width=1\textwidth]{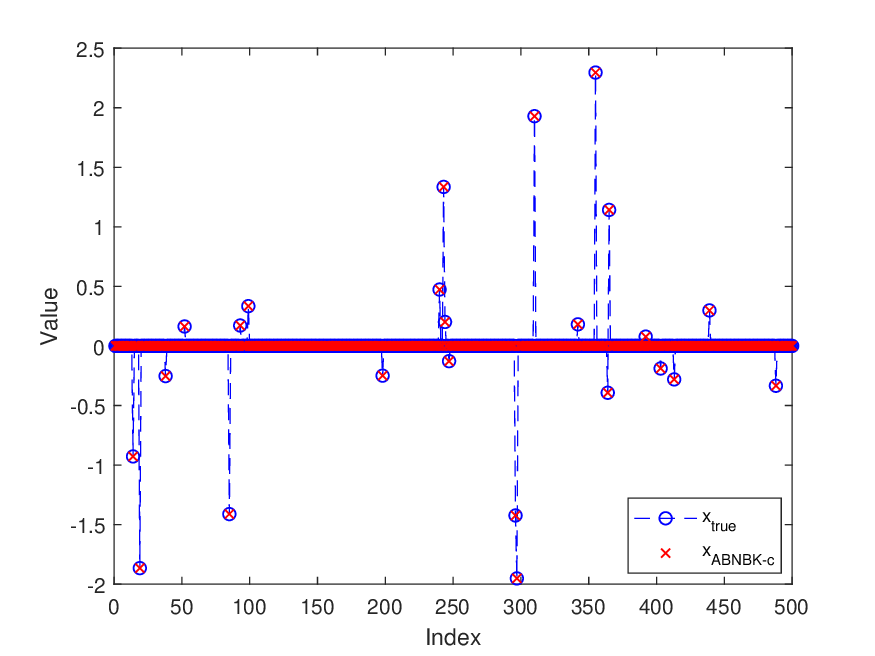}
		\end{minipage}
	}
	\subfigure[ABNBK-a]
	{
		\begin{minipage}[t]{0.48\linewidth}
			\centering 
			\includegraphics[width=1\textwidth]{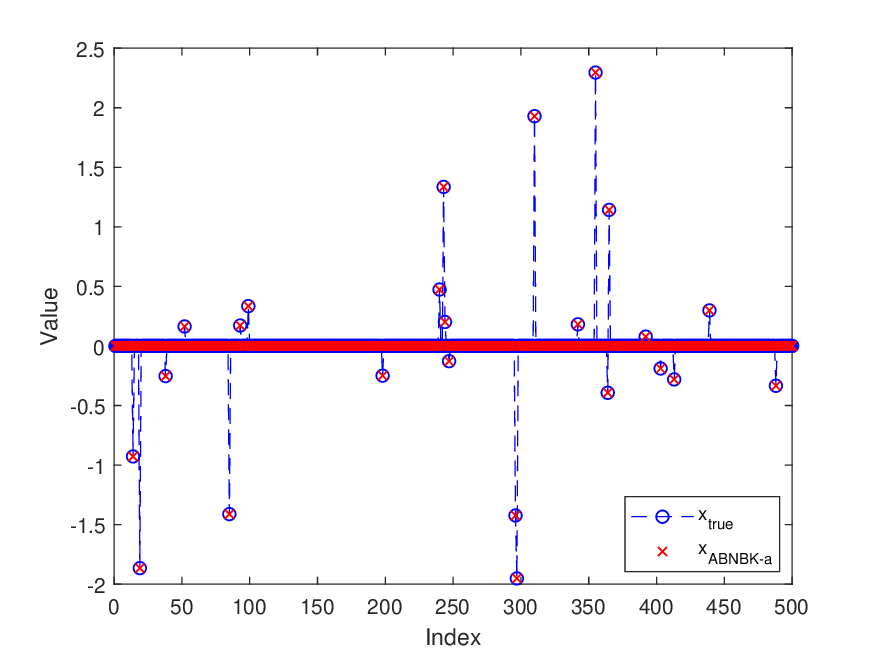}
		\end{minipage}
	} 
  \caption{The exact signal and recovered signals for Example \ref{ex1:DCTmatrix} with $sp=0.05$.}
  \label{fig:xRecoversp005DCT}
\end{figure}

\newpage
\section{Conclusions} \label{secconclu_abnbk} 
In this paper, an averaging block nonlinear Bregman-Kaczmarz method is developed for the nonlinear sparse signal recovery problem. The convergence theory of the averaging block nonlinear Bregman-Kaczmarz method is established under the classical local tangential cone conditions. Moreover, the upper bound of the convergence rate of the averaging block nonlinear Bregman-Kaczmarz method with both constant stepsizes and adaptive stepsizes is given, respectively. 
Numerical experiments demonstrate the efficiency and robustness of the proposed method, which has faster convergence speed and fewer computing time than existing nonlinear Bregman-Kaczmarz methods. \\
 
\noindent \textbf{Funding} This work was supported by National Natural Science Foundation of China (No. 12471357).\\

\bibliographystyle{plain}	 
\bibliography{refsABBregNKa}

\end{document}